\documentclass[final,leqno]{siamltex}

\usepackage{hyperref}
\usepackage{graphicx}
\usepackage{bm}
\usepackage{amsmath}
\usepackage{xcolor}
\usepackage{amssymb}
\usepackage{amsfonts,stackengine}
\usepackage{marginnote}
\usepackage{fancybox}
\usepackage{color}
\usepackage{bm}
\usepackage{cancel}
\usepackage{wrapfig}
\usepackage[notcite, notref]{showkeys} 
\usepackage{lipsum}
\usepackage{epstopdf}
\usepackage{stmaryrd}
\usepackage{mathrsfs}
\usepackage{psfrag}
\usepackage{caption}
\usepackage{mathtools}
\usepackage{amsopn}
\usepackage{placeins}
\usepackage{enumerate}
\usepackage[normalem]{ulem}
\usepackage{enumitem,kantlipsum}

\newcommand{\EO}[1]{\noindent{\textcolor{black}{#1}}}
\newtheorem{remark}{Remark}[section]

\newcommand{\T}{\mathscr{T}}
\newcommand{\V}{\mathbb{V}}

\title{
Error estimates for fractional semilinear optimal control on Lipschitz polytopes\thanks{Submitted to the editors \today. The author is partially supported by ANID through FONDECYT project 1220156.}}
\author{Enrique Ot\'arola\thanks{Departamento de Matem\'atica, Universidad T\'ecnica Federico Santa Mar\'ia, Valpara\'iso, Chile ({\tt enrique.otarola@usm.cl}).}
}

\begin{document}
\maketitle

\begin{abstract}
\EO{We adopt the \emph{integral} definition of the fractional Laplace operator and analyze solution techniques for fractional, semilinear, and elliptic optimal control problems posed on Lipschitz polytopes. We consider two strategies of discretization: a semidiscrete scheme where the  admissible control set is not discretized and a fully discrete scheme where such a set is discretized with piecewise constant functions. As an instrumental step, we derive error estimates for finite element discretizations of fractional semilinear elliptic partial differential equations (PDEs) on quasi-uniform and graded meshes. With these estimates at hand, we derive error bounds for the semidiscrete scheme and improve the ones that are available in the literature for the fully discrete scheme.}
\end{abstract}

\begin{keywords}
optimal control, fractional diffusion, integral fractional Laplacian, semilinear equations, regularity estimates, \EO{finite elements, graded meshes,} a priori error estimates.
\end{keywords}

\begin{AMS}
35R11,     
49J20,     
49M25,     
65K10,     
65N15,     
65N30.     
\end{AMS}

\section{Introduction}
In this work we are interested in the analysis of finite element discretization techniques for a distributed optimal control problem involving a fractional, semilinear, and elliptic \EO{PDE}. To make the discussion precise, we let $\Omega \subset \mathbb{R}^n$, with $n \geq 2$, be an open, bounded, and Lipschitz polytope. Given $\alpha > 0$, the so-called regularization parameter, and $L: \Omega \times \mathbb{R} \rightarrow \mathbb{R}$, a Carath\'eodory function of class $C^2$ with respect to the second variable, we introduce the cost functional
\begin{equation}
 \label{eq:cost_functional}
 J(u,z) := \int_{\Omega} L (x,u(x)) \mathrm{d}x + \frac{\alpha}{2} \int_{\Omega} |z(x)|^2  \mathrm{d}x;
\end{equation}
further assumptions on $L$ will be deferred until section \ref{sec:assumption}. \EO{We are then interesting in finding}
%
$
 \min  J(u,z)
$
subject to the \emph{fractional, semilinear}, and \emph{elliptic} PDE
\begin{equation}
\label{eq:state_equation}
(-\Delta)^s u + a(\cdot,u) = z \textrm{ in } \Omega, 
\qquad u = 0 \textrm{ in } \Omega^c,
\qquad
\EO{\Omega^c = \mathbb{R}^n\setminus \Omega,}
\qquad
\EO{s \in (0,1),}
\end{equation}
and the \emph{control constraints}
$
 \mathfrak{a} \leq z(x) \leq \mathfrak{b}
$
for a.e.~$x \in \Omega$; \EO{$\mathfrak{a},\mathfrak{b} \in \mathbb{R}$ are such that $\mathfrak{a} < \mathfrak{b}$. In \eqref{eq:state_equation}, $(-\Delta)^s$ corresponds to the \emph{integral} fractional Laplacian and $a$ denotes a nonlinear function; see section \ref{sec:assumption} for assumptions on $a$.
We will refer to the previously introduced optimization problem as the \emph{fractional semilinear optimal control problem}.}

The development and analysis of solution techniques for problems involving suitable definitions of fractional diffusion is a relatively new but rapidly growing area of research. We refer the interested reader to \cite{MR3893441,MR4189291} for a complete overview of the available results and limitations. In contrast to these advances, the study of numerical methods for optimal control problems involving fractional diffusion has not been as developed. Restricting ourselves to problems that consider the \emph{spectral definition}, we mention \cite{MR3429730,MR4015150,MR3702421} within the linear--quadratic scenario, \cite{MR4066856,MR3739306} for sparse PDE-constrained optimization, and \cite{MR3939497} for bilinear optimal control. Concerning problems involving the \emph{integral} definition of fractional diffusion, we mention \cite{MR3990191,glusaotarola} for the linear--quadratic case, \EO{\cite{bilinear_otarola} for bilinear optimal control,} and \cite{MR4358465} for semilinear optimal control. \EO{We conclude this paragraph by mentioning the advances in parameter identification for nonlocal/fractional operators of \cite{MR3472639,MR4369059,MR4441219} and the ones, at the continuous level, in fractional semilinear optimal control \cite{MR4358465,semilinear}, external optimal control for fractional diffusion \cite{MR3988258}, and fractional optimal control with state constraints \cite{MR4119494}.}

This paper extends the recent work \cite{MR4358465} in several directions. In what follows, we briefly detail our main contributions and improvements on the available theory:
\begin{enumerate}[leftmargin=0.5cm]
 \item \emph{Finite element discretizations of fractional semilinear PDEs}:
 \EO{We discretize fractional semilinear PDEs using continuous piecewise linear finite elements and derive in Theorems \ref{thm:error_estimates_state_equation} and \ref{thm:error_estimates_state_equation_2} error bounds on Lipschitz polytopes. We improve upon \cite{MR4358465} and extend the global estimates in \cite{MR4283703} to a semilinear setting. As an instrumental step, we derive regularity estimates. For $s \geq n/(4(n-1))$, the estimates of Theorems \ref{thm:error_estimates_state_equation} and \ref{thm:error_estimates_state_equation_2} are improved in Theorem \ref{thm:error_estimates_state_equation_graded} by considering suitable graded meshes but at the expense of requiring that $\Omega$ satisfies an \emph{exterior ball condition}. The restriction $s \geq n/(4(n-1))$ guarantees that we can utilize \cite[estimate (3.17)]{MR4283703}.}
  \item \emph{Regularity estimates for optimal variables}: We derive regularity estimates for an optimal triplet $(\bar u, \bar p, \bar z)$  in H\"older and Sobolev spaces; see Theorems \ref{thm:regularity_space}, \ref{thm:regularity_space_Sobolev}, and \ref{thm:regularity_space_Sobolev_2}. The results on Sobolev spaces hold under the assumption that $\Omega$ satisfies an \emph{exterior ball condition} and improves upon \cite{MR4358465}, where $\partial \Omega \in C^{\infty}$.
  \item \emph{Finite element discretizations for the optimal control problem}.  \EO{For the fully discrete scheme of \cite{MR4358465}, we derive a nearly-optimal estimate for the error commited within the approximation of an optimal control variable on Lipschitz polytopes. This improves upon \cite{MR4358465}, where $\partial \Omega \in C^{\infty}$. 
  In addition, we propose a semidiscrete scheme based on the variational discretization approach \cite{Hinze:05} and perform an error analysis on conforming and shape regular families of simplicial triangulations. Under the assumption that $\Omega$ satisfies an \emph{exterior ball condition}, in Theorem  \ref{thm:error_estimate_control_va_graded} we improve the aforementioned error analysis by considering suitable graded meshes.}
\end{enumerate}

\section{Notation and preliminaries}
\label{sec:notation_and_preliminaries}
We begin this section by fixing notation and the setting in which we will operate. Throughout this work $n \geq 2$ and $\Omega \subset \mathbb{R}^n$ is an open, bounded, and Lipschitz polytope.
We will denote by $\Omega^c$ the complement of $\Omega$. If $\mathfrak{X}$ and $\mathfrak{Z}$ are Banach function spaces, we write $\mathfrak{X}  \hookrightarrow \mathfrak{Z}$ to denote that $\mathfrak{X}$ is continuously embedded in $\mathfrak{Z}$. We denote by $\mathfrak{X}'$ and $\| \cdot \|_{\mathfrak{X}}$ the dual and the norm of $\mathfrak{X}$, respectively. We denote by $\langle \cdot,\cdot \rangle_{\mathfrak{X}',\mathfrak{X}}$ the duality paring between $\mathfrak{X}'$ and $\mathfrak{X}$ and simply write $\langle \cdot,\cdot \rangle$ when the spaces $\mathfrak{X}'$ and $\mathfrak{X}$ are clear from the context. Let $\{ x_n \}_{n=1}^{\infty}$ be a sequence in $\mathfrak{X}$. We denote by $x_n \rightarrow x$ and $x_n \rightharpoonup x$ the  strong and weak convergence, respectively, of $\{ x_n \}_{n=1}^{\infty}$ to $x$. The relation ${\sf a} \lesssim {\sf b}$ indicates that ${\sf a} \leq C {\sf b}$, with a positive constant $C$ that does not depend on either ${\sf a}$, ${\sf b}$, or the discretization parameters but that might depend on $s$, $n$, and $\Omega$. The value of $C$ might change at each occurrence.

\subsection{Function spaces}
\label{sec:function_spaces}
For any $s \geq 0$, we define $H^s(\mathbb{R}^n)$, the fractional Sobolev space of order $s$ over $\mathbb{R}^n$, by \cite[Definition 15.7]{Tartar} 
\[
 H^s(\mathbb{R}^n) := \left \{ v \in L^2(\mathbb{R}^n): (1 + |\xi|^2)^{\frac{s}{2}} \mathcal{F}(v) \in L^2(\mathbb{R}^n) \right \},
\]
endowed with the norm $\| v \|_{H^s(\mathbb{R}^n)}:= \| (1 + |\xi|^2)^{\frac{s}{2}} \mathcal{F}(v)\|_{L^2(\mathbb{R}^n)}$. 
\EO{We define $\tilde H^s(\Omega)$ as the closure of $C_0^{\infty}(\Omega)$ in $H^s(\mathbb{R}^n)$ and immediately notice that it can be equivalently characterized as the following space of zero-extension functions \cite[Theorem 3.29]{McLean}:
\begin{equation}
\label{eq:H_tilde_s}
\tilde H^s(\Omega) = \{v|_{\Omega}: v \in H^s(\mathbb{R}^n), \textrm{ supp } v \subset \overline\Omega\}.
\end{equation}
We endow $\tilde{H}^{s}(\Omega)$ with the following inner product and norm \cite[page 75]{McLean}:}
\begin{equation*}
\label{eq:inner_product}
( v, w )_{\tilde{H}^{s}(\Omega)} :=\int_{\mathbb{R}^{d}}\int_{\mathbb{R}^{d}}\frac{(v(x) - v(y))(w(x) - w(y))}{|x - y|^{n+2s}}\mathrm{d}x\mathrm{d}y,
\quad
\| v \|_{\tilde{H}^s(\Omega)}:= ( v, v )^{\frac{1}{2}}_{\tilde{H}^{s}(\Omega)}.
\end{equation*}
We denote by $H^{-s}(\Omega)$ the dual space of $\tilde H^s(\Omega)$.

We conclude this section with the following Sobolev embedding results.
\begin{proposition}[embedding results]
\EO{Let $s \in (0,1)$.} If $\mathfrak{q} \in [1,2n/(n-2s)]$, then $H^s(\Omega) \hookrightarrow L^{\mathfrak{q}}(\Omega)$. If $\mathfrak{q} \in [1,2n/(n-2s))$, then $H^s(\Omega) \hookrightarrow L^{\mathfrak{q}}(\Omega)$ is compact.
\label{pro:Sobolev_embedding}
\end{proposition}
\begin{proof}
A proof of \EO{$H^s(\Omega) \hookrightarrow L^{\mathfrak{q}}(\Omega)$} can be found in \cite[Theorem 7.34]{MR2424078} while the compactness of the embedding \EO{for $\mathfrak{q} < 2n/(n-2s)$} follows from \cite[Theorem 6.3]{MR2424078}.
\end{proof}


\subsection{The fractional Laplace operator}
\label{section:the_fractional_laplacian}
\EO{For $s \in (0,1)$ and smooth functions $w:\mathbb{R}^n \rightarrow \mathbb{R}$, there are several equivalent definitions of $(-\Delta)^s$ in $\mathbb{R}^n$. In fact, $(-\Delta)^s$ can be naturally defined} via Fourier transform: $\mathcal{F}( (-\Delta)^s w) (\xi) = | \xi |^{2s} \mathcal{F}(w) (\xi)$. Equivalently, $(-\Delta)^s$ can be defined by means of the following pointwise formula:
\begin{equation}
 (-\Delta)^s w(x) = C(n,s)\, \mathrm{ p.v. } \int_{\mathbb{R}^n} \frac{w(x) - w(y)}{|x-y|^{n+2s}} \mathrm{d}y,
 \qquad
 C(n,s) = \frac{2^{2s} s \Gamma(s+\frac{n}{2})}{\pi^{\frac{n}{2}}\Gamma(1-s)}.
 \label{eq:integral_definition}
\end{equation}
Here, $\textrm{p.v.}$ stands for the Cauchy principal value and $C(n,s)$ is a normalization constant that is introduced \EO{to guarantee that definition \eqref{eq:integral_definition} is equivalent to the one via Fourier transform; see \cite[Chapter 1, \S 1]{Landkof} for details.} We notice that \eqref{eq:integral_definition} clearly displays the nonlocal structure of $ (-\Delta)^s$: computing $(-\Delta)^s w(x)$ requires the values of $w$ at points arbitrarily far away from $x$. In addition to these two definitions, several other \emph{equivalent} definitions of $(-\Delta)^s$ in $\mathbb{R}^n$ are available in the literature \cite{MR3613319}---for instance, the ones based on the Balakrishnan formula 
and a suitable harmonic extension.

In \EO{bounded domains there are also several definitions of $(-\Delta)^s$. For} functions supported in $\bar \Omega$, we may utilize the \emph{integral representation} \eqref{eq:integral_definition} to define $(-\Delta)^s$. This gives rise to the so-called \emph{restricted} or \emph{integral} fractional Laplacian, which, from now on, we shall simply refer to as the \emph{integral fractional Laplacian}. Notice that we have materialized a zero Dirichlet condition by restricting $(-\Delta)^s$ to acting only on functions that are zero outside $\Omega$. \EO{We must mention that in bounded domains, and in addition to the integral fractional Laplacian, there are, at least, two other \emph{nonequivalent} definitions of nonlocal operators related to the fractional Laplacian: the \emph{regional} and the \emph{spectral} fractional Laplacians; see the discussion in \cite[\S 6]{MR3503820}.}

To present suitable weak formulations for problems involving 
$(-\Delta)^s$, we \EO{introduce 
\begin{equation}
\label{eq:bilinear_form}
\mathcal{A}:\tilde{H}^{s}(\Omega)\times\tilde{H}^{s}(\Omega) \rightarrow \mathbb{R},
\qquad
 \mathcal{A}(v,w) := \tfrac{C(n,s)}{2} ( v, w )_{\tilde{H}^{s}(\Omega)}.
\end{equation}
We notice that $\mathcal{A}$ is just a multiple of the inner product in $\tilde{H}^{s}(\Omega)$. In particular, $\mathcal{A}$ is bilinear and bounded.}
We denote by $\| \cdot \|_s$ the norm induced by $\mathcal{A}$:
$
\| v \|_s := \sqrt{\mathcal{A}(v,v)}
= \mathfrak{C}(n,s) |v|_{\tilde{H}^s(\Omega)}
$
with
$
\mathfrak{C}(n,s) = \sqrt{C(n,s)/2}.
$



\subsection{Assumptions}\label{sec:assumption}

\EO{The following set of assumptions allows us to perform an analysis for the fractional semilinear optimal control problem including existence of solutions and first and necessary and sufficient second order optimality conditions. We must immediately mention that depending on the property under interest, the requirements differ and would have to be specified anew. To avoid this, we list a set of assumptions to hold throughout the article.}

\begin{enumerate}[label=(A.\arabic*)]
\item \label{A1} $a:\Omega\times \mathbb{R}\rightarrow  \mathbb{R}$ is a Carath\'eodory function of class $C^2$ with respect to the second variable and $a(\cdot,0)\in L^{r}(\Omega)$ for $r>n/2s$.

\item \label{A2} $\frac{\partial a}{\partial u}(x,u)\geq 0$ for a.e.~$x\in\Omega$ and for all $u \in \mathbb{R}$.

\item \label{A3} For all $\mathfrak{m}>0$, there exists a positive constant $C_{\mathfrak{m}}$ such that 
\begin{equation*}
\sum_{i=1}^{2}\left|\frac{\partial^{i} a}{\partial u^{i} }(x,u)\right|\leq C_{\mathfrak{m}},
\qquad
\left|\frac{\partial^{2} a}{\partial u^{2} }(x,v)- \frac{\partial^{2} a}{\partial u^{2} }(x,w)\right|\leq C_{\mathfrak{m}} |v-w|
\end{equation*}
for a.e.~$x\in \Omega$ and $u,v,w \in [-\mathfrak{m},\mathfrak{m}]$.
\end{enumerate}

\begin{enumerate}[label=(B.\arabic*)]
\item \label{B1} $L: \Omega \times \mathbb{R} \rightarrow \mathbb{R}$ is a Carath\'eodory function of class $C^2$ with respect to the second variable and $L(\cdot,0) \in L^1(\Omega)$.

\item \label{B2} For all $\mathfrak{m}>0$, there exist $\psi_{\mathfrak{m}}, \phi_{\mathfrak{m}} \in L^{r}(\Omega)$, with $r>n/2s$, such that
\begin{equation*}
\left|
\frac{\partial L}{\partial u}(x,u)
\right|
\leq \psi_{\mathfrak{m}}(x),
\qquad
\left|
\frac{\partial^2 L}{\partial u^2}(x,u)
\right|
\leq \phi_{\mathfrak{m}}(x),
\end{equation*}
for a.e.~$x\in \Omega$ and $u \in [-\mathfrak{m},\mathfrak{m}]$.
\end{enumerate}

We \EO{briefly comment on the set of assumptions \ref{A1}--\ref{B2}. Assumption \ref{A2} allows us to apply the theory of \emph{monotone operators} for \eqref{eq:state_equation} while the fact that $a(\cdot,0) \in L^r(\Omega)$ in \ref{A1} guarantees that solutions to fractional semilinear PDEs are bounded in $L^{\infty}(\Omega)$. Assumptions on the first and second derivatives of $a, L$ are needed to perform first and second order optimality conditions, respectively.}
%
%
%
%
%
\section{Fractional semilinear PDEs}
\label{sec:state_equation}

In this section, \EO{we introduce a weak formulation for a fractional semilinear PDE and review results regarding the well-posedness of such a formulation and regularity estimates for its solution.}

\subsection{\EO{Weak formulation}}
Let \EO{$s \in (0,1)$, $f \in H^{-s}(\Omega)$, and $a = a(x,u) : \Omega \times \mathbb{R} \rightarrow \mathbb{R}$ be a Carath\'eodory function that is monotone increasing in $u$. Assume that, for every $\mathfrak{m}>0$, there exits
\begin{equation}
\varphi_{\mathfrak{m}} \in L^{\mathfrak{t}}(\Omega):
\quad
|a(x,u)| \leq | \varphi_{\mathfrak{m}}(x)|
~\textrm{a.e.}~x \in \Omega,~u \in [-\mathfrak{m},\mathfrak{m}],
\quad
\mathfrak{t} =2n/(n+2s).
\label{eq:assumption_on_phi_state_equation}
\end{equation}
Within this setting, we introduce the following weak formulation:} 
\begin{equation}
\label{eq:weak_semilinear_pde_integral}
u \in \tilde H^s(\Omega):
\quad
\mathcal{A}(u,v)  +  \langle a(\cdot,u),v \rangle = \langle f , v \rangle 
\quad \forall v \in \tilde H^{s}(\Omega).
\end{equation}

The following existence and uniqueness result follows from \cite[Theorem 3.1]{MR4358465}.

\begin{theorem}[\EO{well-posedness of fractional semilinear PDEs}]
Let $n \geq 2$,  $s \in (0,1)$, and $r>n/2s$. 
If $f \in L^{r}(\Omega)$, $a$ satisfies \eqref{eq:assumption_on_phi_state_equation}, and $a(\cdot,0) \in L^r(\Omega)$, 
then problem \eqref{eq:weak_semilinear_pde_integral} admits a unique solution $u \in \tilde H^s(\Omega) \cap L^{\infty}(\Omega)$. In addition, we have the bound
\begin{equation}
 | u |_{H^s(\mathbb{R}^n)} + \| u \|_{L^{\infty}(\Omega)} \lesssim \| f - a(\cdot,0) \|_{L^{r}(\Omega)},
 \label{eq:stability_integral}
\end{equation}
with a hidden constant that is independent of $u$, $a$, and $f$.
\label{thm:stata_equation_well_posedness_integral}
\end{theorem}

\subsection{Regularity estimates}
In order to derive a priori error estimates for suitable finite element discretizations of problem \eqref{eq:weak_semilinear_pde_integral}, \EO{it is of fundamental importance to understand the regularity properties} of the solution to \eqref{eq:weak_semilinear_pde_integral}. 

\subsubsection{\EO{The linear case}}
We begin our studies by providing some basic regularity results for the linear case $a \equiv 0$.

\begin{proposition}[H\"older regularity]
Let $s \in (0,1)$, and let \EO{$\Omega$ be a bounded Lipschitz domain} satisfying an exterior ball condition. Let $\mathsf{u}$ be the solution to $(-\Delta)^s \mathsf{u} = \mathsf{f}$ in $\Omega$ and $\mathsf{u} = 0$ in $\Omega^c$. If 
$
\mathsf{f} \in L^{\infty}(\Omega),
$ 
then $\mathsf{u} \in C^s(\mathbb{R}^n)$ and 
\begin{equation}
 \| \mathsf{u} \|_{C^s(\mathbb{R}^n)} \lesssim \| \mathsf{f} \|_{L^{\infty}(\Omega)},
 \label{eq:regularity_state_Holder}
  \end{equation}
with a hidden constant that only depends on $\Omega$ and $s$.
  \label{pro:state_regularity_Holder}
 \end{proposition}
\begin{proof}
See \cite[Proposition 1.1]{MR3168912}\EO{.}
\end{proof}

\EO{The following remark presents an example that is essential.}

\begin{remark}[optimal regularity]\rm
Let $\Omega = B(0,1) \subset \mathbb{R}^n$ and $f \equiv 1$. Within this setting, the solution to $(-\Delta)^s \mathsf{u} = \mathsf{f}$ in $\Omega$ and $\mathsf{u} = 0$ in $\Omega^c$ is given by \cite{MR137148}
\[
 \mathsf{u}(x) = \frac{\Gamma(\frac{n}{2})}{2^{2s}\Gamma(\frac{n+2s}{2})\Gamma(1+s)}\left(1-|x|^2\right)^s_{+}, 
 \qquad 
 t_{+} = \max \{ t,0\}.
\]
The solution $\mathsf{u} \in C^s(\bar \Omega)$ but it does not belong to $C^{\alpha}(\bar \Omega)$ for any $\alpha>s$. In this sense, the $C^s(\bar \Omega)$-regularity result stated in Proposition \ref{pro:state_regularity_Holder} is optimal \cite[page 276]{MR3168912}.
\label{rk:example}
\end{remark}

We present a regularity result \EO{in} Sobolev spaces; see \cite[Theorem 2.1]{MR4283703} and \cite{MR4530901}.

\begin{proposition}[Sobolev regularity]
Let $s \in (0,1)$, and let $\Omega$ be a bounded Lipschitz domain. Let $\mathsf{u}$ be the solution to $(-\Delta)^s \mathsf{u} = \mathsf{f}$ in $\Omega$ and $\mathsf{u} = 0$ in $\Omega^c$. If $\mathsf{f} \in L^2(\Omega)$, \EO{then there exist constants $C, \zeta >0$ such that} $\mathsf{u} \in H^{s + \theta-\epsilon}(\Omega)$, where $\theta = \tfrac{1}{2}$ for $\tfrac{1}{2}<s<1$ and $\theta = s - \epsilon >0$ for $0 < s \leq \tfrac{1}{2}$; $0<\epsilon <s$. In addition, we have
\begin{equation}
\begin{aligned}
   \| \mathsf{u} \|_{H^{2s -2\epsilon}(\Omega)} 
  &
  \leq \EO{C} 
  \epsilon^{-\frac{1}{2}-\EO{\zeta}} \| \mathsf{f} \|_{L^2(\Omega)}, \quad s \in (0, \tfrac{1}{2}],
   \quad \forall 0<\epsilon<s,
   \\
  \| \mathsf{u} \|_{H^{s + \frac{1}{2}-\epsilon}(\Omega)} 
  &
  \leq \EO{C} 
  \epsilon^{-\frac{1}{2}} \| \mathsf{f} \|_{L^2(\Omega)}, \quad s \in (\tfrac{1}{2},1),
   \quad \forall 0<\epsilon<s+\tfrac{1}{2}.
  \end{aligned}
 \label{eq:regularity_state_Lipschitz_new}
  \end{equation}
The constant $C$ is independent of $\epsilon$ but \EO{depends} on $\Omega$, $n$, and $s$.
  \label{pro:state_regularity_Lipschitz_new}
 \end{proposition}

The following comments are in order \EO{with} \cite{MR4283703}. First, the \emph{Lipschitz} assumption on $\Omega$ is \emph{optimal} in the sense that if $\Omega$ was a $C^{\infty}$ domain, then no further regularity could be inferred. \EO{Thus, reentrant corners play no role in the global regularity of solutions: the boundary behavior $\mathsf{u}(x) \approx \mathrm{dist}(x,\partial \Omega)^{s}v(x)$, with $v$ being H\"older continuous up to $\partial \Omega$, dominates any point singularities that could originate from them. We refer the interested reader to \cite{corner_singularities} for further details.} Second, in general the smoothness of the right-hand side cannot make solutions any smoother than $\cap_{\epsilon >0} \tilde H^{s+1/2-\epsilon}(\Omega)$. \EO{These two comments are illustrated within the setting of the example in Remark \ref{rk:example}.}

For $s \in (0,\tfrac{1}{2})$, the regularity properties of Proposition \ref{pro:state_regularity_Lipschitz_new} can be improved but
at the expense of considering a smoother domain $\Omega$ and a smoother forcing term $\mathsf{f}$.

\begin{proposition}[Sobolev regularity]
Let $s \in (0,\tfrac12)$, and let $\Omega$ be a bounded Lipschitz domain satisfying an exterior ball condition. Let $\mathsf{u}$ \EO{solve} $(-\Delta)^s \mathsf{u} = \mathsf{f}$ in $\Omega$ and $\mathsf{u} = 0$ in $\Omega^c$. If $\mathsf{f} \in C^{\frac{1}{2}-s}(\bar \Omega)$, then $\mathsf{u} \in H^{s + \frac{1}{2}-\epsilon}(\Omega)$ 
\EO{and}
\begin{equation}
  \| \mathsf{u} \|_{H^{s + \frac{1}{2}-\epsilon}(\Omega)} 
  \lesssim 
  \epsilon^{-1} \| \mathsf{f} \|_{C^{\frac{1}{2}-s}(\bar \Omega)}, \quad s \in (0,\tfrac{1}{2}),
   \quad \forall 0<\epsilon<s+\tfrac{1}{2},
 \label{eq:regularity_state_Lipschitz_old}
  \end{equation}
with a  hidden constant that is independent of $\epsilon$ but \EO{depends} on $\Omega$, $n$, and $s$.
  \label{pro:state_regularity_Lipschitz_old}
 \end{proposition}
\begin{proof}
See \cite[Theorem 3.3]{MR3893441}.
\end{proof}


\subsubsection{\EO{The semilinear case}}
We now derive a regularity result in H\"older spaces for \EO{the solution to the fractional semilinear PDE \eqref{eq:weak_semilinear_pde_integral}.}

\begin{theorem}[H\"older regularity]
\label{thm:regularity_space_state_equation_integral_Holder}
Let $n \geq 2$ and $s \in (0,1)$. Let $\Omega$ be a bounded Lipschitz domain satisfying an exterior ball condition. \EO{Let $a$ be as in the statement of}
Theorem \ref{thm:stata_equation_well_posedness_integral}. Assume, in addition, that $a = a(x,u)$ is locally Lipschitz in $u$, uniformly for $x \in \Omega$. If $f, a(\cdot,0) \in L^{\infty}(\Omega)$, then $u \in C^s(\mathbb{R}^n)$. \EO{In addition, we have the bound}
\[
\| u \|_{C^s(\mathbb{R}^n)}  \lesssim \| f -a(\cdot,0)  \|_{L^{\infty}(\Omega)},
\]
with a hidden constant that depends on $\Omega$ and $s$.
\end{theorem}
\begin{proof}
\EO{We begin the proof by noticing that, for every $\mathfrak{m}>0$ and $v \in [-\mathfrak{m},\mathfrak{m}]$, we have $|a(x,v)| \leq |a(x,0)| + |a(x,v) - a(x,0)| \leq |a(x,0)| + C_{\mathcal{L}} |v|$ for a.e.~$x \in \Omega$. Here, $C_{\mathcal{L}}$ denotes the Lipschitz constant of $a$. This proves that $a(\cdot,u) \in L^{\infty}(\Omega)$. 
}
The fact that $u \in C^s(\mathbb{R}^n)$ thus follows immediately from Proposition \ref{pro:state_regularity_Holder}. In addition, we have
\[
\| u \|_{C^s(\mathbb{R}^n)} 
\lesssim
\| f - a(\cdot,0) \|_{L^{\infty}(\Omega)}
+
\| u  \|_{L^{\infty}(\Omega)}
\lesssim
\| f - a(\cdot,0)  \|_{L^{\infty}(\Omega)}.
\]
To obtain the \EO{first estimate we have utilized again the fact that $a = a(x,u)$ is locally Lipschitz in $u$, uniformly for $x \in \Omega$. The second estimate follows directly from Theorem \ref{thm:stata_equation_well_posedness_integral}: $\| u  \|_{L^{\infty}(\Omega)} \lesssim \| f - a(\cdot,0)  \|_{L^{\infty}(\Omega)}$.} This concludes the proof.
\end{proof}

In view of the regularity requirements on the forcing term $\mathsf{f}$ stated in Proposition \ref{pro:state_regularity_Lipschitz_old}, the following remark, which provides necessary and sufficient conditions for the boundedness of a Nemitskii operator in H\"older spaces, is of particular importance.
\begin{remark}[The Nemitskii operator in H\"older spaces]
\rm
Let $g$ be a real-valued function defined on $\Omega \times \mathbb{R}$. We introduce the Nemitskii operator induced by $g$ as follows: $G(x)(u):= g(x,u(x))$ with $x \in \Omega$ and $u$ varying in a suitable space of real-valued functions defined on $\Omega$. $G$ maps $C^{0,\varrho}(\bar \Omega)$, with $\varrho \in (0,1]$, into itself if $g$ satisfies the following condition: For every $\mathfrak{m}>0$, there exists $\mathfrak{M} = \mathfrak{M}(\mathfrak{m})>0$ such that
 \begin{equation}
 \label{eq:Nemitskii}
  |g(x,u) - g(y,v)| \leq \mathfrak{M} \left\{ |x-y|^{\varrho} + \mathfrak{m}^{-1} |u-v| \right\}
 \end{equation}
for all $x,y \in \bar \Omega$  and for all $u,v \in \mathbb{R}$ such that $|u|, |v| \leq \mathfrak{m}$. In other words, we demand that $g = g(x,u)$ be H\"older continuous in $x$, uniformly for $u$ in bounded intervals of $\mathbb{R}$, and locally Lipschitz in $u$, uniformly for $x \in \bar \Omega$. This condition is also \emph{necessary} 
when $\Omega$ is a general open and bounded set of $\mathbb{R}^n$ \cite[Theorem 1.1]{MR1325579}.
 \label{rk:Nemitskii} 
\end{remark}

\begin{theorem}[Sobolev regularity]
\label{thm:regularity_space_state_equation_integral_Sobolev}
Let $s \in [\tfrac{1}{4},1)$ and $n \geq 2$. Let $\Omega$ be a bounded Lipschitz domain such that it satisfies an exterior ball condition for $s < \tfrac{1}{2}$. Let $a$ be as in the statement of Theorem \ref{thm:stata_equation_well_posedness_integral}. Assume, in addition, that $a=a(x,u)$ is locally Lipschitz in $u$, uniformly for $x \in \Omega$, $a(\cdot,0) \in L^{2}(\Omega)$ for $s \in [\frac{1}{2},1)$, $a(\cdot,0) \in L^{\infty}(\Omega)$ for $s \in [\tfrac{1}{4},\frac{1}{2})$, and that $a$ satisfies \eqref{eq:Nemitskii} with 
\begin{equation}
\varrho = \tfrac{1}{2}-s~\mathrm{ for }~s \in [\tfrac{1}{4},\tfrac{1}{2}).
\label{eq:zeta}
\end{equation}
If $f \in L^r(\Omega)$, for $r>n/2s$ and, in addition,
\begin{equation}
\label{eq:f}
f \in C^{\frac{1}{2}-s}(\bar \Omega)~\mathrm{ for }~s \in [\tfrac{1}{4},\tfrac{1}{2}),
\qquad
f \in L^{2}(\Omega)~\mathrm{ for }~s \in [ \tfrac{1}{2},1)
\end{equation}
then, we have that $u \in H^{s+\frac{1}{2}-\epsilon}(\Omega)$ for every 
$\epsilon \in (0,\epsilon_{\star})$; the precise value of $\epsilon_{\star}$ is described in estimates \eqref{eq:reg_u_s_121}, \eqref{eq:reg_u_s_12}, and \eqref{eq:reg_u_s_141}.
\end{theorem}
\begin{proof}
We consider three cases.

\emph{Case} 1: $s \in (\tfrac{1}{2},1)$. \EO{Since $a=a(x,u)$ is locally Lipschitz in $u$, uniformly for $x \in \Omega$, and $a(\cdot,0) \in L^{2}(\Omega)$, we deduce that $a(\cdot,u) \in L^2(\Omega)$ and thus that} $f - a(\cdot,u) \in L^2(\Omega)$. We can thus apply the regularity results of Proposition \ref{pro:state_regularity_Lipschitz_new} to arrive at
\begin{equation}
 \begin{aligned}
\| u \|_{H^{s+\frac12-\epsilon}(\Omega)} 
& \lesssim
\epsilon^{-\frac{1}{2}}
\left( 
\| f - a(\cdot,0) \|_{L^2(\Omega)} + \| u \|_{L^2(\Omega)}
\right)
\\
& \lesssim
\epsilon^{-\frac{1}{2}}
\| f - a(\cdot,0) \|_{L^2(\Omega)} 
\quad
\forall 0<\epsilon<s+\tfrac{1}{2},
\quad
s \in (\tfrac{1}{2},1),
\label{eq:reg_u_s_121}
\end{aligned}
\end{equation}
upon utilizing again that $a=a(x,u)$ is locally Lipschitz in $u$, uniformly for $x \in \Omega$, and the stability estimate $\| u \|_{H^s(\mathbb{R}^n)} \lesssim \| f - a(\cdot,0) \|_{L^2(\Omega)}$, which follows from Theorem \ref{thm:stata_equation_well_posedness_integral}. In both estimates the hidden constant is independent of $\epsilon$.

\emph{Case} 2: $s = \tfrac{1}{2}$. \EO{In this case, we have that $u \in H^{s+\frac{1}{2}-2\epsilon}(\Omega)$. This follows from the arguments  elaborated in the previous case and Proposition \ref{pro:state_regularity_Lipschitz_new}. In addition, we have}
%
\begin{equation}
\| u \|_{H^{s+\frac12-2\epsilon}(\Omega)} 
\lesssim 
\epsilon^{-\frac{1}{2} - \EO{\zeta}}  
\| f - a(\cdot,0) \|_{L^{2}(\Omega)}
\quad
\forall 0<2\epsilon<s + \tfrac{1}{2},
\quad
s = \tfrac{1}{2}.
\label{eq:reg_u_s_12}
\end{equation}
The hidden constant is independent of $\epsilon$.

\emph{Case} 3: $s \in [\tfrac{1}{4},\tfrac{1}{2})$. \EO{Within this setting,} Theorem \ref{thm:regularity_space_state_equation_integral_Holder} guarantees that $u \in C^s(\mathbb{R}^n)$. We can thus invoke the fact that $a$ satisfies \eqref{eq:Nemitskii} with $\varrho = \tfrac{1}{2} - s$ to arrive at $f - a(\cdot,u) \in C^{\varrho}(\bar \Omega)$ (see Remark \ref{rk:Nemitskii}).
We are thus in position to invoke the results of Proposition \ref{pro:state_regularity_Lipschitz_old} to deduce that $u \in H^{s+1/2-\epsilon}(\Omega)$ together with the bounds
\begin{equation}
\begin{aligned}
\| u \|_{H^{s+\frac{1}{2}-\epsilon}(\Omega)} 
& \lesssim 
\epsilon^{-1} 
\left[     
\|f \|_{C^{\frac12-s}(\bar \Omega)}
+
\| a(\cdot, u) \|_{C^{\frac12-s}(\bar \Omega)}
\right]
\\
& \lesssim 
\epsilon^{-1} 
\left[   
1
+
\|f \|_{C^{\frac12-s}(\bar \Omega)}
+
\| u \|_{C^{s}(\bar \Omega)}
\right]
\,\,
\forall \epsilon \in (0, s +\tfrac{1}{2}),
\,
s \in [\tfrac{1}{4},\tfrac{1}{2}),
\end{aligned}
\label{eq:reg_u_s_141}
\end{equation}
with a hidden constant that is independent of $\epsilon$.
\end{proof}

We now present a regularity result for $s \in (0,\frac{1}{2})$. When $s\in [\tfrac{1}{4},\tfrac{1}{2})$ the result is weaker than the one obtained in Theorem \ref{thm:regularity_space_state_equation_integral_Sobolev}. Nevertheless, it holds under weaker assumptions on the forcing term $f$, the nonlinear function $a$, \EO{and the domain $\Omega$.}

\begin{theorem}[Sobolev regularity]
\label{thm:regularity_space_state_equation_integral_SobolevII}
Let $s \in (0,\tfrac{1}{2})$ and $n \geq 2$. Let $\Omega$ be a bounded Lipschitz domain. Let $a$ be as in the statement of Theorem \ref{thm:stata_equation_well_posedness_integral}. Assume, in addition, that $a=a(x,u)$ is locally Lipschitz in $u$, uniformly for $x \in \Omega$, and $a(\cdot,0) \in L^{2}(\Omega)$. If $f \in L^r(\Omega) \cap L^2(\Omega)$, for $r>n/2s$, then $u \in H^{2s-2\epsilon}(\Omega)$ for every $\epsilon \in (0,s)$ and 
\[
 \| u \|_{H^{2s-2\epsilon}(\Omega)}
\lesssim
 \epsilon^{-\frac{1}{2} - \zeta}  
\| f - a(\cdot,0) \|_{L^{2}(\Omega)},
\quad
s \in (0,\tfrac{1}{2}),
\quad
\forall 0<\epsilon<s,
\]
where $\zeta$ is as in the statement of Proposition \ref{pro:state_regularity_Lipschitz_new}.
\end{theorem}
\begin{proof}
The proof follows similar arguments to those elaborated in the proof of Theorem \ref{thm:regularity_space_state_equation_integral_Sobolev}. For brevity, we skip the details.
\end{proof}

\section{Fractional semilinear PDE-constrained optimization}
\label{sec:optimal_control_problem_integral}

\EO{We consider the} following weak version of the \emph{fractional semilinear optimal control problem}: Find
\begin{equation}\label{eq:min_integral}
\min \{ J(u,z): (u,z) \in \tilde H^s(\Omega) \times \mathbb{Z}_{ad}\}
\end{equation}
subject to the \emph{fractional, semilinear,} and \emph{elliptic} state equation
\begin{equation}\label{eq:weak_st_eq_integral}
\mathcal{A}( u, v)+(a(\cdot,u),v)_{L^2(\Omega)}=(z,v)_{L^2(\Omega)} \quad \forall v \in \tilde H^s(\Omega).
\end{equation}
\EO{Here, $\mathbb{Z}_{ad}:=\{ v \in L^2(\Omega):  \mathfrak{a} \leq v(x) \leq \mathfrak{b}~\text{a.e.}~x \in \Omega \}$ and $\mathfrak{a},\mathfrak{b} \in \mathbb{R}$ are such that $\mathfrak{a} < \mathfrak{b}$.}

In \EO{view of the assumptions on $a$, Theorem \ref{thm:stata_equation_well_posedness_integral} guarantees} that \eqref{eq:weak_st_eq_integral} admits a unique solution $u \in \tilde H^s(\Omega) \cap L^{\infty}(\Omega)$. We thus introduce the control to state map $\mathcal{S}: L^{r}(\Omega) \rightarrow \tilde H^s(\Omega) \cap L^{\infty}(\Omega)$ which, given a control $z$, associates to it the unique state $u$ that solves \eqref{eq:weak_st_eq_integral}.  We also introduce 
$j:\mathbb{Z}_{ad} \rightarrow \mathbb{R}$ by the relation $j(z)=J(\mathcal{S}z,z)$.

The \EO{existence of an optimal solution $\bar{z} \in \mathbb{Z}_{ad}$ follows from \cite[Theorem 4.1]{MR4358465}.}

%
%
\subsection{First order necessary optimality conditions}\label{sec:1st_order}
\EO{In this section, we state} first order necessary optimality conditions for \eqref{eq:min_integral}--\eqref{eq:weak_st_eq_integral}. We must immediately mention that, since \eqref{eq:min_integral}--\eqref{eq:weak_st_eq_integral} is not convex, we distinguish between local and global solutions and present optimality conditions in the context of local solutions in \EO{$L^2(\Omega)$.}

To \EO{formulate first order optimality conditions, 
we introduce the adjoint state $p \in \tilde H^s(\Omega) \cap L^{\infty}(\Omega)$ as the solution to the \emph{adjoint equation}}
\begin{equation}\label{eq:adj_eq_integral}
\mathcal{A}(v,p) + \left(\frac{\partial a}{\partial u}(\cdot,u)p,v\right)_{L^2(\Omega)} = \left(\frac{\partial L}{\partial u}(\cdot,u),v\right)_{L^2(\Omega)}
\quad \forall v \in \tilde H^s(\Omega).
\end{equation}
\EO{The well-posedness of the adjoint problem \eqref{eq:adj_eq_integral} follows from assumptions \textnormal{\ref{A2}} and \textnormal{\ref{B2}}, which guarantee that $\partial a/\partial u (x,u) \geq 0$, for a.e.~$x \in \Omega$ and for all $u \in \mathbb{R}$, and that $\partial L/ \partial u(\cdot,u) \in L^r(\Omega)$, for some $r>n/2s$.}

First \EO{order optimality conditions for our optimal control problem read as follows: If $\bar z \in  \mathbb{Z}_{ad}$ denotes a locally optimal control for \eqref{eq:min_integral}--\eqref{eq:weak_st_eq_integral}, then \cite[Theorem 4.4]{MR4358465}}
\begin{equation}\label{eq:var_ineq_integral}
\left( \bar p +\alpha \bar{z},z-\bar{z}\right)_{L^2(\Omega)}\geq 0 \quad \forall z\in \mathbb{Z}_{ad},
\end{equation}
where $\bar p \in \tilde H^s(\Omega) \cap L^{\infty}(\Omega)$ denotes the solution to \eqref{eq:adj_eq_integral} with $u$ replaced by $\bar{u} = \mathcal{S} \bar z$.

Define \EO{$\Pi_{[\mathfrak{a},\mathfrak{b}]} : L^1(\Omega) \rightarrow  \mathbb{Z}_{ad}$ by $\Pi_{[\mathfrak{a},\mathfrak{b}]}(v) := \min\{ \mathfrak{b}, \max\{ v, \mathfrak{a} \} \}$ a.e.~in~$\Omega$.}
The following projection formula is of fundamental importance to study regularity estimates. If $\bar{z} \in \mathbb{Z}_{ad}$ denotes a locally optimal control for \eqref{eq:min_integral}--\eqref{eq:weak_st_eq_integral}, then
 \cite[page 217]{MR2583281}
\begin{equation}\label{eq:projection_control} 
\bar{z}(x):=\Pi_{[\mathfrak{a},\mathfrak{b}]}(-\alpha^{-1}\bar{p}(x)) \textrm{ a.e.}~x \in \Omega.
\end{equation}
Since $\bar p \in \tilde H^s(\Omega) \cap L^{\infty}(\Omega)$ and $s \in (0,1)$, it is immediate that $\bar z \in H^{s}(\Omega) \cap L^{\infty}(\Omega)$.

\subsection{Second order optimality conditions}\label{sec:2nd_order}
In \eqref{eq:var_ineq_integral} we stated first order \emph{necessary} optimality conditions. Since our optimal control problem is not convex, \emph{sufficiency} requires the use of second order optimality conditions. \EO{To elaborate on these conditions, we introduce some preliminary concepts. Let $\bar{z} \in \mathbb{Z}_{ad}$ satisfies \eqref{eq:var_ineq_integral}.  We define $\bar{\mathfrak{p}} :=  \bar p + \alpha \bar z$ 
and the \emph{cone of critical directions}:}
\begin{equation}
C_{\bar{z}}:=\{v\in L^2(\Omega): \eqref{eq:sign_cond} \text{ holds and } \bar{\mathfrak{p}}(x) \neq 0 \implies v(x) = 0 \},
\label{eq:Cz}
\end{equation}
where condition \eqref{eq:sign_cond} reads as follows:
\begin{equation}
\label{eq:sign_cond}
v(x)
\geq 0  \text{ a.e.}~x\in\Omega \text{ if } \bar{z}(x)=\mathfrak{a},
\qquad
v(x) \leq 0 \text{ a.e.}~x\in\Omega \text{ if } \bar{z}(x)=\mathfrak{b}.
\end{equation}
 
Second \EO{order \emph{necessary} optimality conditions for problem \eqref{eq:min_integral}--\eqref{eq:weak_st_eq_integral} can be found in \cite[Theorem 4.6]{MR4358465}. Reference \cite{MR4358465} also provide sufficient second order conditions, which hold under the extra assumptions that $n\in\{2,3\}$ and $s>n/4$. By exploiting the fact that $\mathbb{Z}_{ad} \subset L^{\infty}(\Omega)$, we remove these assumptions and improve upon \cite{MR4358465}.}

\begin{theorem}[second order sufficient optimality conditions]
\label{thm:suff_opt_cond}
Let \EO{$n \geq 2$ and $s \in (0,1)$. Let $\bar u \in \tilde H^s(\Omega)$, $\bar p \in \tilde H^s(\Omega)$, and $\bar{z} \in \mathbb{Z}_{ad}$ satisfy \eqref{eq:weak_st_eq_integral}, \eqref{eq:adj_eq_integral}, and \eqref{eq:var_ineq_integral}. If
$
j''(\bar{z})v^2 > 0 
$
for all $v\in C_{\bar{z}} \setminus \{ 0 \}$, then there exists $\kappa > 0$ and $\EO{\delta} >0$ such that
\begin{equation}
\label{eq:quadratic_growing}
j( z) \geq j(\bar z) + \tfrac{\kappa}{2} \| z - \bar  z\|_{L^2(\Omega)}^2
\end{equation}
for every $z \in \mathbb{Z}_{ad}$ such that $\| \bar z - z \|_{L^2(\Omega)} \leq \EO{\delta}$.}
\end{theorem}
\begin{proof}
\EO{The proof basically follows the arguments elaborated in the proof of \cite[Theorem 4.7]{MR4358465}. The only argument that needs to be modified is the one that allows the convergence of the sequences $\{ \hat{u}_k\}_{k \in \mathbb{N}}$ and $\{ \hat{p}_k \}_{\mathbb{N}}$ to $\bar{u}$ and $\bar{p}$, respectively, in $\tilde{H}^s(\Omega) \cap L^{\infty}(\Omega)$ as $k \uparrow \infty$. Observe that the convergence $\hat{z}_k \rightarrow \bar{z}$ in $L^2(\Omega)$ as $k \uparrow \infty$ combined with the fact that $\{ \hat{z}_k \}_{k\in \mathbb{N}}$ is uniformly bounded in $L^{\infty}(\Omega)$ allow us to conclude that $\hat{z}_k \rightarrow \bar{z}$ in $L^{\iota}(\Omega)$ as $k \uparrow \infty$ for every $\iota \in (2,\infty)$. With such a convergence result at hand, we rewrite the problem that $\bar{u} - \hat{u}_k$ solves as a linear problem \cite[Theorem 4.16]{MR2583281}, invoke the assumptions on $a$, and utilize the stability bound \eqref{eq:stability_integral} to deduce that $\hat{u}_k \rightarrow \bar{u}$ in $\tilde{H}^s(\Omega) \cap L^{\infty}(\Omega)$ as $k \uparrow \infty$ without further assumptions on $n$ and $s$. Let us now observe that assumptions \ref{A3} and \ref{B2} guarantee that
\[
 \left\| \frac{\partial a}{\partial u}(\cdot,\bar{u}) -\frac{\partial a}{\partial u}(\cdot,\hat{u}_k) \right\|_{L^r(\Omega)}
 +
 \,\,\,
 \left\| \frac{\partial L}{\partial u}(\cdot,\bar{u}) -\frac{\partial L}{\partial u}(\cdot,\hat{u}_k) \right\|_{L^r(\Omega)}
 \rightarrow 0,
 \qquad
 k \uparrow \infty.
\]
Consequently, $\hat{p}_k \rightarrow \bar{p}$ in $\tilde{H}^s(\Omega) \cap L^{\infty}(\Omega)$ as $k \uparrow \infty$. This concludes the proof.
}
\end{proof}

The \EO{following result is the starting point to derive error estimates for suitable finite element discretizations of the optimal control problem \eqref{eq:min_integral}--\eqref{eq:weak_st_eq_integral}.}

\begin{theorem}[equivalent optimality conditions]
\label{thm:equivalent_opt_cond}
\EO{Let $n \geq 2$ and $s \in (0,1)$. Let $\bar{u} \in \tilde H^s(\Omega)$, $\bar p \in \tilde H^s(\Omega)$, and $\bar z \in \mathbb{Z}_{ad}$ satisfy the first order optimality conditions \eqref{eq:weak_st_eq_integral}, \eqref{eq:adj_eq_integral}, and \eqref{eq:var_ineq_integral}. Then, the following statements are equivalent:
\begin{equation}\label{eq:second_order_equivalent}
j''(\bar{z})v^2 > 0 \quad \forall v\in C_{\bar{z}} \setminus \{ 0 \} 
\iff
\exists \EO{\nu}, \tau >0: 
\,\, 
j''(\bar{z})v^2 \geq \EO{\nu} \|v\|_{L^2(\Omega)}^2 
\quad \forall v \in C_{\bar{z}}^\tau,
\end{equation}
where
$
C_{\bar{z}}^\tau:=\{v\in L^2(\Omega): \eqref{eq:sign_cond} \textnormal{ holds and } |\bar{\mathfrak{p}}(x)|>\tau \implies v(x)=0 \}.
$}
\end{theorem}
\begin{proof}
\EO{The proof of the equivalence \eqref{eq:second_order_equivalent} follows from a combination of the arguments elaborated in the proofs of \cite[Theorem 4.8] {MR4358465} and Theorem \ref{thm:suff_opt_cond}.}
\end{proof}

\subsection{Regularity estimates}
\label{sec:regularity}
In this section, \EO{we derive regularity estimates for an optimal triplet $(\bar u, \bar p, \bar z)$. To accomplish this task, we will assume that, in addition to \ref{A1}--\ref{A3} and \ref{B1}--\ref{B2}, $a$ and $L$ satisfy the following assumptions:}
\begin{enumerate}[label=(C.\arabic*)]
\item \label{C1} 
\EO{For $s \in [\tfrac{1}{4},\tfrac{1}{2})$ and for all $\mathfrak{m}>0$, there exists a positive constant $C_{\mathfrak{m}}$ such that
$
 |a(x,u)| \leq C_{\mathfrak{m}}
$
and
$
| \partial L/\partial u (x,u)| \leq C_{\mathfrak{m}}
$ 
for a.e.~$x \in \Omega$ and $u \in [-\mathfrak{m},\mathfrak{m}].$}
%
%
\end{enumerate}

With \EO{assumption \ref{C1} at hand, we derive a first regularity result in H\"older spaces.}

\begin{theorem}[H\"older regularity]
Let $n \geq 2$ and $s \in [\frac{1}{4},\tfrac{1}{2})$. Let $\Omega$ be a bounded Lipschitz domain satisfying an exterior ball condition. Let $(\bar u, \bar p, \bar z)$
be an optimal triplet. Then, $\bar u \in C^s(\mathbb{R}^n)$, $\bar p \in C^s(\mathbb{R}^n)$, and $\bar z \in C^s(\bar \Omega)$.
\label{thm:regularity_space}
\end{theorem}
\begin{proof}
Since $\bar z - a(\cdot,\bar u) \in L^{\infty}(\Omega)$, we are in position to utilize the regularity results of Proposition \ref{pro:state_regularity_Holder} to obtain that $\bar u \in C^s(\mathbb{R}^n)$ together with 
\begin{equation}
\begin{aligned}
\| \bar u \|_{C^s(\mathbb{R}^n)} 
& \lesssim  \| \bar z - a(\cdot,\bar u)\|_{L^{\infty}(\Omega)} 
\lesssim 
\| \bar z - a(\cdot,0) \|_{L^{\infty}(\Omega)}
+
\| \bar u \|_{L^{\infty}(\Omega)}
\\
& \lesssim
\| \bar z - a(\cdot,0)\|_{L^{\infty}(\Omega)},
\qquad
s \in [\tfrac{1}{4},\tfrac{1}{2}).
\end{aligned}
\label{eq:first_regularity_u}
\end{equation}
In all three estimates the hidden constant is independent of $\bar u$ and $\bar z$ but it depends on $\Omega$ and $s$. To obtain the second estimate in \eqref{eq:first_regularity_u}, we have used that $a = a(x,u)$ is locally Lipschitz in $u$, uniformly for $x \in \Omega$. The third estimate in \eqref{eq:first_regularity_u} follows from \eqref{eq:stability_integral}. To derive a regularity result for $\bar p$, we observe that $\bar p \in L^{\infty}(\Omega)$. This is a consequence of the existence of $r > n/2s$ such that $\partial L/\partial u (\cdot,\bar u) \in L^r(\Omega)$ \EO{(see assumption \ref{B2})} and the fact that $0 \leq \partial a/\partial u(\cdot,\bar u) \in L^{\infty}(\Omega)$ \EO{(see assumption \ref{A3})}. Consequently, assumption \ref{C1} guarantees that $\partial L/\partial u (\cdot,\bar u)  - \partial a/\partial u (\cdot,\bar u) \bar p \in L^{\infty}(\Omega)$. We can thus invoke Proposition \ref{pro:state_regularity_Holder} to conclude that $\bar p \in C^s(\mathbb{R}^n)$ together with
\begin{equation}
\| \bar p \|_{C^s(\mathbb{R}^n)} 
\lesssim  
\left 
\| \frac{\partial L}{\partial u} (\cdot,\bar u)  - \frac{\partial a}{\partial u}(\cdot,\bar u) \bar p 
\right \|_{L^{\infty}(\Omega)} 
\lesssim
\left 
\| \frac{\partial L}{\partial u} (\cdot,\bar u) \right \|_{L^{\infty}(\Omega)}  
+
\left 
\| \bar p 
\right \|_{L^{\infty}(\Omega)},
\label{eq:first_regularity_p}
\end{equation}
upon utilizing the first estimate in \ref{A3}. Finally, the projection formula \eqref{eq:projection_control} and \cite[Theorem A.1]{MR1786735} allow us to conclude that $\bar z \in C^s(\bar \Omega)$ with a similar estimate.
\end{proof}


To \EO{present regularity results on Sobolev spaces we will assume, in addition, that} $a$, $\partial a/ \partial u $, and $\partial L/ \partial u$ satisfy the following assumptions:

\begin{enumerate}[label=(D.\arabic*)]
\item \label{D1} Let $s \geq \tfrac{1}{2}$. For all $\mathfrak{m}>0$ and $u \in [-\mathfrak{m},\mathfrak{m}]$, $a(\cdot,u), \partial L/ \partial u (\cdot,u) \in L^2(\Omega)$.
\item \label{D2} For $s \in [\tfrac{1}{4},\tfrac{1}{2})$, $a$, $\partial a/ \partial u$, and $\partial L/ \partial u$ satisfy \eqref{eq:Nemitskii} with $\varrho = \tfrac{1}{2} - s$.
\end{enumerate}

\begin{theorem}[Sobolev regularity]
Let $n \geq 2$, $s \in [\tfrac{1}{4},1)$, and let $\Omega$ be a bounded Lipschitz domain such that it satisfies an exterior ball condition for $s<\tfrac{1}{2}$. If $(\bar u, \bar p, \bar z)$
denotes an optimal triplet, then
\begin{equation}
\bar u, \bar p, \bar z \in H^{s + \frac{1}{2} - \epsilon}(\Omega)
\end{equation}
for every $\epsilon \in (0,\epsilon_{\ast})$; $\epsilon_{\ast}>0$ being described in estimates \eqref{eq:reg_u_s121}--\eqref{eq:reg_p_1412}.
\label{thm:regularity_space_Sobolev}
\end{theorem}
\begin{proof} 
We consider three cases.

\emph{Case} 1: $s \in (\tfrac{1}{2},1)$. Since $\bar z - a(\cdot,\bar u) \in L^2(\Omega)$, we can invoke the regularity results of Proposition \ref{pro:state_regularity_Lipschitz_new} to deduce $\bar u \in H ^{s+1/2-\epsilon}(\Omega)$ for every $0<\epsilon<s+\frac{1}{2}$. In addition, 
\begin{equation}
\begin{aligned}
\| \bar u \|_{H^{s+\frac{1}{2}-\epsilon}(\Omega)} 
& \lesssim \epsilon^{-\frac{1}{2}} \left( \| \bar z - a(\cdot,0) \| _{L^2(\Omega)} +  \| \bar u \| _{L^2(\Omega)}\right)
\\
& 
\lesssim 
\epsilon^{-\frac{1}{2}}  \| \bar z - a(\cdot,0) \| _{L^2(\Omega)}
\quad
\forall \epsilon \in (0,s+\tfrac{1}{2}),
\quad
s \in (\tfrac{1}{2},1),
\end{aligned}
\label{eq:reg_u_s121}
\end{equation}
upon utilizing that \EO{$a = a(x,u)$ is locally Lipschitz in $u$, uniformly for $x \in \Omega$, and the stability bound $| \bar u |_{H^s(\mathbb{R}^n)} \lesssim \| f - a(\cdot,0) \|_{L^2(\Omega)}$.} Similarly, the regularity results of Proposition \ref{pro:state_regularity_Lipschitz_new} applied now to the adjoint equation \eqref{eq:adj_eq_integral} allow us to conclude that $\bar p \in H ^{s+1/2-\epsilon}(\Omega)$, for every $0<\epsilon<s+\frac{1}{2}$, together with the estimates
\begin{equation}
\label{eq:reg_p_s121}
 \begin{aligned}
 \| \bar p \|_{H^{s+\frac{1}{2}-\epsilon}(\Omega)} 
& \lesssim
\epsilon^{-\frac{1}{2}} \left[ \left\| \frac{\partial L}{\partial u} (\cdot,\bar u) \right\|_{L^2( \Omega)} 
+ 
\| \bar p\|_{L^{2}(\Omega)} 
\right]
\\
& \lesssim
\epsilon^{-\frac{1}{2}}  \left\| \frac{\partial L}{\partial u} (\cdot,\bar u) \right\|_{L^2( \Omega)}
\quad
\forall \epsilon \in (0,s+\tfrac{1}{2}),
\quad
s \in (\tfrac{1}{2},1).
\end{aligned}
\end{equation}
%
On the basis of the projection formula \eqref{eq:projection_control}, an application of \cite[Theorem 1]{MR1173747} reveals that
$\bar z \in H^{s+1/2 - \epsilon}(\Omega)$, for every $0<\epsilon<s+\frac{1}{2}$, with a similar estimate.

\emph{Case} 2: $s = \tfrac{1}{2}$. \EO{An immediate application of 
Proposition \ref{pro:state_regularity_Lipschitz_new} reveals that} 
\begin{equation}
\| \bar u \|_{H^{s+\frac{1}{2}-2\epsilon}(\Omega)} \lesssim \epsilon^{-\frac{1}{2}-\zeta}\|  \bar z - a(\cdot,0) \|_{L^2(\Omega)}
\quad
\forall \, 2\epsilon \in (0,s+\tfrac{1}{2}),
\quad
s = \tfrac{1}{2}.
\label{eq:reg_u_s12}
\end{equation} 
We now notice that \ref{A3}, \ref{D1}, and the well-posedness of the adjoint equation \eqref{eq:adj_eq_integral} reveal that $\partial L/\partial L(\cdot,\bar u) - \partial a/\partial u(\cdot, \bar u)\bar p \in L ^{2}(\Omega)$. Invoke Proposition \ref{pro:state_regularity_Lipschitz_new} to obtain
\begin{equation}
\begin{aligned}
\| \bar p \|_{H^{s+\frac{1}{2}-2\epsilon}(\Omega)} 
& \lesssim 
\epsilon^{-\frac{1}{2} + \zeta}
\left\| \frac{\partial L}{\partial u} (\cdot,\bar u) \right\|_{L^2( \Omega)}
\quad
\forall \,2\epsilon \in (0,s+\tfrac{1}{2}),
\quad
s = \tfrac{1}{2}.
\end{aligned}
\label{eq:reg_p_s12}
\end{equation}
The projection formula \eqref{eq:projection_control} combined with an application of \cite[Theorem 1]{MR1173747} reveal that $\bar z \in H^{s+1/2 - 2\epsilon}(\Omega)$, for every $2\epsilon \in (0,s+\tfrac{1}{2})$, with a similar estimate.

\emph{Case} 3: $s \in [\tfrac{1}{4},\tfrac{1}{2})$. \EO{Since $\bar z \in C^{1/2-s}(\bar \Omega)$, which follows from $1/2 - s \leq s$ and Theorem \ref{thm:regularity_space}, we can invoke assumption \ref{D2} and Proposition \ref{pro:state_regularity_Lipschitz_old} to conclude that}
\begin{equation}
\| \bar u \|_{H^{s+\frac{1}{2}-\epsilon}(\Omega)} \lesssim \epsilon^{-1} \left(1+ \|  \bar z \|_{C^{s}(\bar \Omega)} + \| \bar u \|_{C^{s}(\bar \Omega)} \right)
\quad
\forall 0<\epsilon < s + \tfrac{1}{2},
\quad
s \in [\tfrac{1}{4},\tfrac{1}{2});
\label{eq:reg_u_1412}
\end{equation}
\EO{compare with \eqref{eq:reg_u_s_141}.} We now notice that $\EO{\partial L/\partial u}(\cdot,\bar u) - \partial a/\partial u(\cdot, \bar u)\bar p \in C^{1/2-s}(\bar \Omega)$. In fact, by assumption \ref{D2} we have that $\EO{\partial L/\partial u}(\cdot,\bar u) \in C^{1/2-s}(\bar \Omega)$. On the other hand, Theorem \ref{thm:regularity_space} guarantees that $\bar p \in C^s(\mathbb{R}^n)$. This combined with \ref{D2} allow us to conclude that $\partial a/\partial u(\cdot, \bar u) \bar p \in C^{1/2-s}(\bar \Omega)$; observe that $1/2 - s \leq s$. We thus invoke Proposition \ref{pro:state_regularity_Lipschitz_old} to arrive at $\bar p \in H^{s+1/2-\epsilon}(\Omega)$, for every $0<\epsilon< s + \frac{1}{2}$, together with
\begin{multline}
\| \bar p \|_{H^{s+\frac{1}{2}-\epsilon}(\Omega)} 
\lesssim \epsilon^{-1} \bigg( \left\| \frac{\partial L}{\partial u} (\cdot,\bar u) \right\|_{C^{\frac12 - s}(\bar \Omega)}
+  
\|  \bar p \|_{C^{\frac12 - s}(\bar \Omega)}
\\
+
\|  \bar p \|_{L^{\infty}(\Omega)}
\left[
1+ \|  \bar u \|_{C^{\frac12 - s}(\bar \Omega)}
\right]
\bigg)
\quad 
\forall 0<\epsilon < s + \tfrac{1}{2},
\quad
s \in [\tfrac{1}{4},\tfrac{1}{2}).
\label{eq:reg_p_1412}
\end{multline}
The projection formula \eqref{eq:projection_control} and \cite[Theorem 1]{MR1173747} reveal that $\bar z \in H^{s+1/2 - \epsilon}(\Omega)$, for every $0<\epsilon < s + \tfrac{1}{2}$, with a similar estimate. This concludes the proof.
\end{proof}

We conclude with a regularity result for $s \in (0,\frac{1}{2})$. When $s\in [\tfrac{1}{4},\tfrac{1}{2})$ the result is weaker than the one obtained in Theorem \ref{thm:regularity_space_Sobolev}. However, it holds under weaker assumptions on the control problem data. In fact, in what follows, \emph{we do not operate under assumptions} \ref{C1}, \ref{D1}and \ref{D2}.

\begin{theorem}[Sobolev regularity]
Let $n \geq 2$, $s \in (0,\tfrac{1}{2})$, and let $\Omega$ be a bounded Lipschitz domain. Assume that, for all $\mathsf{m}>0$ and $u \in [-\mathsf{m},\mathsf{m}]$, $a(\cdot,u), \partial L/\partial u (\cdot,u) \in L^2(\Omega)$. If $(\bar u, \bar p, \bar z) \in \tilde H^s(\Omega) \times \tilde H^s(\Omega) \times \mathbb{Z}_{ad}$ denotes an optimal triplet, then
\begin{equation}
\bar u, \bar p, \bar z \in H^{2s - 2\epsilon}(\Omega)
\end{equation}
for every $\epsilon \in (0,s)$.
\label{thm:regularity_space_Sobolev_2}
\end{theorem}
\begin{proof}
The proof follows similar arguments to those elaborated in the proof of Theorem \ref{thm:regularity_space_Sobolev}. For brevity, we skip the details.
\end{proof}

\section{Finite element approximation of fractional semilinear PDEs}
\label{sec:fem}
Let us begin by describing the finite element framework that we will adopt. To avoid technical difficulties, we assume that $\Omega$ is a Lipschitz polytope \EO{so that it can be triangulated exactly}; when needed, we shall additionally assume that it satisfies an exterior ball condition. Let $\mathbb{T}$ be a collection of conforming and simplicial triangulations $\T = \{ T \}$ of $\bar \Omega$, which are obtained by subsequent refinements of an initial mesh. We assume that $\mathbb{T}$ is \emph{shape regular}. \EO{Finally, we define $h_{\T} := \max \{h_T: T \in \T \}$, where $h_T := \mathrm{diam}(T)$.}

Given \EO{a mesh $\T \in \mathbb{T}$, we introduce the basic finite element space} 
$
\V(\T) = \{ v_{\T }\in C^0( \overline {\Omega} ): v_{\T}|_T \in \mathbb{P}_1(T) \ \forall T \in \T,
\,
 v_{\T} = 0 \textrm{ on } \partial \Omega\}.
$
\EO{The following comments are now in order. First, for every $s \in (0,1)$, $\V(\T) \subset \tilde H^s(\Omega)$. Second, we enforce a classical homogeneous Dirichlet boundary condition at \(\partial \Omega\). Observe that discrete functions are trivially extended by zero to $\Omega^c$.} 

\subsection{The discrete problem}
We introduce the following finite element approximation of the semilinear elliptic PDE \eqref{eq:weak_semilinear_pde_integral}: Find $u_{\T} \in \V(\T)$ such that
\begin{equation}
 \label{eq:discrete_semilinear_pde}
  \mathcal{A} (u_{\T},v_{\T})  +  \int_{\Omega} a(x,u_{\T}(x)) v_{\T}(x) \mathrm{d}x = \int_{\Omega} f(x) v_{\T}(x) \mathrm{d}x \quad \forall v_{\T} \in \V(\T).
\end{equation}
\EO{Let $r$, $f$, and $a$ be as in the statement of Theorem \ref{thm:stata_equation_well_posedness_integral}.}
Since the bilinear form $\mathcal{A}$ is coercive and $a$ is monotone increasing in $u$, an application of Brouwer's fixed point theorem yields the existence of a unique solution for problem \eqref{eq:discrete_semilinear_pde}.
In addition, we have the stability bound $\| u_{\T} \|_s \lesssim \| f \|_{H^{-s}(\Omega)}$.

\subsection{Error estimates}
In what follows, we derive error estimates for the proposed finite element scheme. To accomplish this task, we will assume that
\begin{enumerate}[label=(E.\arabic*)]
\item \EO{The nonlinear function $a$ satisfies
$
|a(x,u) - a(x,v)| \leq |\phi(x)| |u-v|
$
for a.e.~$x \in \Omega$ and $u,v \in \mathbb{R}$, where $\phi \in L^{\mathfrak{r}}(\Omega)$ and $\mathfrak{r} = \tfrac{n}{2s}$.}
\label{eq:assumption_on_a_phi}
\end{enumerate}

To simplify the presentation of the derived error bounds, we define
\begin{equation}
\begin{aligned}
\Lambda(f,a)  := 1
+
\|f \|_{C^{\frac{1}{2}-s}(\bar \Omega)}
+
\| f - a(\cdot,0) \|_{L^{\infty}(\Omega)},
\,\,\,
\Sigma(f,a)  := \| f - a(\cdot,0) \|_{L^{2}(\Omega)}.
\end{aligned}
\end{equation}

\begin{theorem}[error estimates]
Let $n \geq 2$ and $s \in (0,1)$. Let $\Omega \subset \mathbb{R}^n$ be a bounded Lipschitz domain. Let $a$ be as in the statement of Theorem \ref{thm:stata_equation_well_posedness_integral}. Assume, in addition, that $a$ satisfies \ref{eq:assumption_on_a_phi}. \EO{Let $u \in \tilde H^s(\Omega)$ and $u_{\T} \in \mathbb{V}(\T)$ be the solutions to \eqref{eq:weak_semilinear_pde_integral} and  \eqref{eq:discrete_semilinear_pde}, respectively.} Then, we have the quasi-best approximation result
\begin{equation}
\label{eq:error_estimate_semilinear_s}
\| u - u_{\T} \|_{s} \lesssim \| u - v_{\T} \|_{s} \quad \forall v_{\T} \in \mathbb{V}(\T).
\end{equation}
If, in addition, $s \in [\tfrac{1}{4},1)$, $\Omega$ satisfies an exterior ball condition for $s<\tfrac{1}{2}$, $a = a(x,u)$ is locally Lipschitz in $u$, uniformly for $x \in \Omega$, $a(\cdot,0) \in L^{2}(\Omega)$ for $s \in [\frac{1}{2},1)$, $a(\cdot,0) \in L^{\infty}(\Omega)$ for $s \in [\tfrac{1}{4},\frac{1}{2})$, $a$ satisfies \eqref{eq:Nemitskii} with $\varrho$ as in  \eqref{eq:zeta}, and $f \in L^r(\Omega)$ $(r>n/2s)$ satisfies \eqref{eq:f}, then we have the quasi-optimal a priori error estimates
\begin{align}
\| u - u_{\T} \|_{s} & \lesssim h_{\T}^{\frac{1}{2} } | \log h_{\T} | \Lambda(f,a),
\,\,\,\,
s \in [\tfrac{1}{4},\tfrac{1}{2}),
\label{eq:error_estimate_semilinear_s_final_1412}
\\
\| u - u_{\T} \|_{s} & \lesssim  h_{\T}^{\frac{1}{2}}|\log h_{\T}|^{\frac{3}{2} + \zeta} \Sigma(f,a),
\,\,\,\,
s = \tfrac{1}{2},
\label{eq:error_estimate_semilinear_s_final_12}
\\
\| u - u_{\T} \|_{s} & \lesssim h_{\T}^{\frac{1}{2} } | \log h_{\T} |^{\frac{1}{2}} \Sigma(f,a),
\,\,\,\,
s \in (\tfrac{1}{2},1).
\label{eq:error_estimate_semilinear_s_final_121}
\end{align} 
Here, \EO{$\zeta$} is as in the statement of Proposition \ref{pro:state_regularity_Lipschitz_new}. \EO{Let $\vartheta := \min \{ s , \tfrac{1}{2}\}$.} If, in addition, \ref{eq:assumption_on_a_phi} holds with $\mathfrak{r}$ \EO{replaced by $\mathfrak{v} := n/s$}, then we have the $L^2(\Omega)$-error estimates:
\begin{align}
\| u - u_{\T} \|_{L^2(\Omega)} \lesssim h_{\T}^{\vartheta + \frac{1}{2}}| \log h_{\T} |^{\frac{3}{2}+\zeta} \Lambda(f,a),
\,\,\,\,
s \in [\tfrac{1}{4},\tfrac{1}{2}),
\label{eq:error_estimate_semilinear_s_final_L2_1412}
\\
\| u - u_{\T} \|_{L^2(\Omega)} \lesssim h_{\T}^{\vartheta + \frac{1}{2}}| \log h_{\T} |^{2\left(\frac{3}{2} + \zeta\right)}\Sigma(f,a),
\,\,\,\,
s = \tfrac{1}{2},
\label{eq:error_estimate_semilinear_s_final_L2_12}
\\
\| u - u_{\T} \|_{L^2(\Omega)} \lesssim h_{\T}^{\vartheta + \frac{1}{2}}| \log h_{\T} |
\Sigma(f,a),
\,\,\,\,
s \in (\tfrac{1}{2},1).
\label{eq:error_estimate_semilinear_s_final_L2_121}
\end{align}
In all estimates the hidden constant is independent of $u$, $u_{\T}$, and $h_{\T}$.
\label{thm:error_estimates_state_equation}
\end{theorem}
\begin{proof}
The proof of \eqref{eq:error_estimate_semilinear_s} follows from the monotonicity of \EO{$a = a(x,u)$} in the second variable, Galerkin orthogonality, assumption \ref{eq:assumption_on_a_phi}, and the Sobolev embedding \EO{of Proposition \ref{pro:Sobolev_embedding}}, $H^s(\Omega) \hookrightarrow L^{\mathfrak{q}}(\Omega)$, which holds for every $\mathfrak{q} \in [1, 2n/(n-2s)]$:
\begin{align*}
\| u - u_{\T} \|_s^2 
& 
\leq 
\mathcal{A}(u - u_{\T},u - u_{\T})  
+ 
(a(\cdot,u) - a(\cdot,u_{\T}), u - u_{\T})_{L^2(\Omega)}
\\
& =  
\mathcal{A}(u - u_{\T},u - v_{\T})  + (a(\cdot,u) - a(\cdot,u_{\T}), u - v_{\T})_{L^2(\Omega)}
\\
&
\leq
\EO{\| u - u_{\T} \|_s  \| u - v_{\T}\|_s + \| \phi \|_{L^{\mathfrak{r}}(\Omega)}\| u - u_{\T} \|_{L^{\mathfrak{q}}(\Omega)}\| u - v_{\T} \|_{L^{\mathfrak{q}}(\Omega)}}
\\
& 
\leq \| u - u_{\T} \|_s  \| u - v_{\T}\|_s 
[ 1 
+
C
\| \phi  \|_{L^{\EO{\mathfrak{r}}}(\Omega)} 
], 
\quad
v_{\T} \in \mathbb{V}(\T), 
\quad
 C > 0.
\end{align*}
\EO{Here, $\mathfrak{r} = n/2s$ and $\mathfrak{q} = 2n/(n-2s)$. Observe that $\mathfrak{r}^{-1} + \mathfrak{q}^{-1} + \mathfrak{q}^{-1} = 1$.}

Assume now that $\Omega$ satisfies an exterior ball condition for $s < \tfrac{1}{2}$ so that we have at hand the results of Theorem \ref{thm:regularity_space_state_equation_integral_Sobolev}. Let $s \in (0,1) \setminus \{ \tfrac{1}{2} \}$. To obtain   \eqref{eq:error_estimate_semilinear_s_final_1412} and \eqref{eq:error_estimate_semilinear_s_final_121}, we bound $\| u - v_{\T} \|_s$ in \eqref{eq:error_estimate_semilinear_s} on the basis of two ingredients. The first one is the bound that is used to prove that $\tilde H^s(\Omega) = H_0^s(\Omega)$ for $s \in (0,1) \setminus \{ \tfrac{1}{2} \}$ \cite[Theorem 3.33]{McLean}:
$
\| u -v_{\T} \|_s \lesssim \| u - v_{\T} \|_{H^s(\Omega)}
$
for all $v_{\T} \in \mathbb{V}(\T)$
and
$
s \in (0,1) \setminus \{ \tfrac{1}{2} \}.
$
The second ingredient is the localization of fractional order Sobolev seminorms  \cite{MR1752263,MR1930387}:
\[
|v|^2_{H^s(\Omega)} \leq \sum_{T} \left[ \int_{T} \int_{S_T} \frac{|v(x) - v(y)|^2}{|x-y|^{n+2s}} \mathrm{d}y \mathrm{d}x + \frac{\mathfrak{c}(n,\sigma)}{ s h_T^{2s} } \| v \|^2_{L^2(T)} \right], 
\quad
s \in (0,1),
\]
for $v \in H^s(\Omega)$. Here, \EO{$S_T = \cup \{ T' \in \T: T' \cap T \neq \emptyset \}$, $\mathfrak{c}(n,\sigma)$ denotes a positive constant, and $\sigma$ is the shape regularity coefficient of the family $\mathbb{T}$.} With these two ingredients at hand, the rest of the proof relies on utilizing interpolation error estimates for the Scott--Zhang operator \cite[Proposition 3.6]{MR3893441}, \cite[Proposition 3.1]{MR4283703} on the basis of the regularity results obtained in Theorem \ref{thm:regularity_space_state_equation_integral_Sobolev}. Since the regularity estimate \eqref{eq:reg_u_s_121} depends on $\epsilon$ as $\epsilon^{-1/2}$, we obtain, for every $\epsilon \in (0,\tfrac{1}{2})$, the error estimate
\[
\| u - u_{\T} \|_s \lesssim h_{\T}^{\frac{1}{2}} \epsilon^{-\frac{1}{2}} h_{\T}^{-\epsilon}
\| f - a(\cdot,0) \|_{L^2(\Omega)}, 
\qquad s \in (\tfrac{1}{2},1),
\]
with a hidden constant independent of $\epsilon$. \EO{In view of $a^{\frac{1}{\ln a}} = e$ for $a\in \mathbb{R}_{+} \setminus \{1\}$,} we thus set $\epsilon = |\log h_{\T}|^{-1}$ to arrive at the error estimate \eqref{eq:error_estimate_semilinear_s_final_121}. The error estimate \eqref{eq:error_estimate_semilinear_s_final_1412} follows similar arguments upon utilizing the regularity estimate \eqref{eq:reg_u_s_141}. Let us now analyze the special case $s = \tfrac{1}{2}$ and derive the error estimate \eqref{eq:error_estimate_semilinear_s_final_12}. To accomplish this task, we utilize a fractional Hardy inequality and an interpolation error estimate for the Scott--Zhang operator: Let $t \in (\tfrac{1}{2},1)$ and $\delta \in (0,t-\tfrac{1}{2})$, then \cite[inequality (3.11)]{MR4283703}
\begin{equation}
 \| v -\Pi_{\T} v \|_{\frac{1}{2}} \lesssim \delta^{-1} h_{\T}^{t - \frac{1}{2} - \delta}| v |_{H^{t}(\Omega)},
 \quad v \in H^t(\Omega).
 \label{eq:no_Hardy}
\end{equation}
We thus invoke the regularity estimate \eqref{eq:reg_u_s_12} and utilize the previous estimate with $v = u$ and $t = 1 - \EO{2}\epsilon$
to conclude that
\begin{equation}
\begin{aligned}
 \| u - u_{\T} \|_{\frac{1}{2}}  
\lesssim h_{\T}^{\frac{1}{2}} \delta^{-1} h_{\T}^{-\delta} \epsilon^{-\frac{1}{2} -\zeta} h_{\T}^{-\EO{2\epsilon}} 
\EO{\Sigma(f,a)}
 \lesssim 
 h_{\T}^{\frac{1}{2}}|\log h_{\T}|^{\frac{3}{2} + \zeta} 
\EO{\Sigma(f,a)},
 \end{aligned}
\end{equation}
upon taking $\varepsilon =  \delta = |\log h_{\T}|^{-1}$.
 
The error estimate in $L^2(\Omega)$ follows from duality. Define 
$ 0 \leq  \chi \in L^{\frac{n}{s}}(\Omega)$ by
\[
\chi(x) = \frac{a(x,u(x)) - a(x,u_{\T}(x))}{u(x) - u_{\T}(x)}
~\mathrm{if}~u(x) \neq u_{\T}(x),
\quad
\chi(x) = 0 
~\mathrm{if}~u(x) = u_{\T}(x).
\]
Let $\mathfrak{z} \in \tilde H^s(\Omega)$ be the solution to
$
\mathcal{A}(v,\mathfrak{z}) + (\chi \mathfrak{z} ,v)_{L^2(\Omega)} = \langle \mathfrak{f} , v \rangle
$
for all $v \in \tilde H^s(\Omega)$; $\mathfrak{f} \in H^{-s}(\Omega)$. Let $\mathfrak{z}_{\T}$ be the finite element approximation of $\mathfrak{z}$ within $\mathbb{V}(\T)$. \EO{Thus, 
\begin{multline}
\langle \mathfrak{f} , u - u_{\T} \rangle  = \mathcal{A}(u-u_{\T},\mathfrak{z}) + (\chi \mathfrak{z}, u - u_{\T})_{L^2(\Omega)} = 
\mathcal{A}(u-u_{\T},\mathfrak{z} - \mathfrak{z}_{\T})
\\
 + \mathcal{A}(u-u_{\T},\mathfrak{z}_{\T})
+ (\chi \mathfrak{z}, u - u_{\T})_{L^2(\Omega)} 
= \mathcal{A}(u-u_{\T},\mathfrak{z} - \mathfrak{z}_{\T})
+ (a(\cdot,u) - a(\cdot,u_{\T}), \mathfrak{z} - \mathfrak{z}_{\T})_{L^2(\Omega)} 
\\
\leq
\| u - u_{\T} \|_s \| \mathfrak{z} - \mathfrak{z}_{\T} \|_s + \| \phi \|_{L^{\mathfrak{r}}(\Omega)} \| u - u_{\T} \|_{L^{\mathfrak{q}}(\Omega)} \| \mathfrak{z} - \mathfrak{z}_{\T} \|_{L^{\mathfrak{q}}(\Omega)},
\label{eq:aux_for_estimate_in_L2}
\end{multline}
where $\mathfrak{r} = n/2s$ and $\mathfrak{q} = 2n/(n-2s)$.}
Set $\mathfrak{f} = u - u_{\T} \in L^2(\Omega)$. 
Since $\phi \in L^{\frac{n}{s}}(\Omega)$, $\chi  \mathfrak{z} \in L^2(\Omega)$. We thus invoke Proposition \ref{pro:state_regularity_Lipschitz_new} to obtain the regularity estimate
\begin{equation}
 \| \mathfrak{z} \|_{H^{s+\theta - \epsilon}(\Omega)}  \lesssim \epsilon^{-\xi}\|  u - u_{\T} \|_{L^2(\Omega)}
\quad
\forall 0<\epsilon<s,
\quad
\theta = \min \{ s -\epsilon, \tfrac{1}{2}\},
\label{eq:f=u-uh}
\end{equation}
where $\xi = \tfrac{1}{2} $ if $s \in (\tfrac{1}{2},1)$ and $\xi = \tfrac{1}{2} + \zeta$ if $s \in (0, \tfrac{1}{2}]$.  If $s \neq \tfrac{1}{2}$, we thus obtain
\[
 \| \mathfrak{z} - \mathfrak{z}_{\T} \|_s 
 \lesssim 
 h_{\T}^{\theta-\epsilon} |\mathfrak{z}|_{H^{s+\theta-\epsilon}(\Omega)} 
 \lesssim
 \epsilon^{-\xi} h_{\T}^{\theta-\epsilon} \| u - u_{\T}  \|_{L^{2}(\Omega)}.
\]
Set $\epsilon = |\log h_{\T}|^{-1}$ to conclude that
$
  \| \mathfrak{z} - \mathfrak{z}_{\T} \|_s \lesssim
   h_{\T}^{\vartheta}| \log h_{\T}|^{\xi} \| u - \EO{u_{\T}}  \|_{L^{2}(\Omega)},
$
where $\vartheta = \min \{s,\tfrac{1}{2} \}$. We now invoke \eqref{eq:aux_for_estimate_in_L2} and the bound \eqref{eq:error_estimate_semilinear_s_final_121} to obtain, for $s \in (\tfrac{1}{2},1)$,
\begin{equation*}
\begin{aligned}
 \| u - u_{\T} \|^2_{L^2(\Omega)} & \lesssim  \| u - u_{\T} \|_s \| \mathfrak{z} - \mathfrak{z}_{\T} \|_s 
 \lesssim h_{\T}^{ \frac{1}{2}+\vartheta} | \log h_{\T} | \| u - u_{\T}  \|_{L^{2}(\Omega)}
\EO{\Sigma(f,a).}
 \end{aligned}
\end{equation*}
The case $s \in [\tfrac{1}{4},\tfrac{1}{2}]$ follows similar arguments. This concludes the proof.
\end{proof}

\begin{remark}[error estimates]
\rm
The energy-norm error bounds \eqref{eq:error_estimate_semilinear_s_final_1412}-- \eqref{eq:error_estimate_semilinear_s_final_121} improve the ones recently obtained in \cite[Theorem 5.2, estimate (5.6)]{MR4358465}: the factor \EO{$h^{-\epsilon}$} in \cite[estimate (5.6)]{MR4358465}, where $\epsilon >0$ is arbitrarily small, has been removed. We also mention that the derived error estimates are in agreement with respect to regularity. The $L^2(\Omega)$-norm error estimates  \eqref{eq:error_estimate_semilinear_s_final_L2_1412}--\eqref{eq:error_estimate_semilinear_s_final_L2_121} read, up to logarithm factors, as follows: 
\[
\| u - u_{\T} \|_{L^2(\Omega)} \lesssim h_{\T}^{s + \frac{1}{2}},
\quad
s \in [\tfrac{1}{4},\tfrac{1}{2}],
\qquad
\| u - u_{\T} \|_{L^2(\Omega)} \lesssim h_{\T},
\quad
s \in (\tfrac{1}{2},1).
\]
Both bounds improve the one derived in \cite[Theorem 5.2, estimate (5.7)]{MR4358465}. In addition, if $s \in [\tfrac{1}{4},\tfrac{1}{2}]$, the error bound is in agreement with respect to regularity. In contrast, when $s \in (\tfrac{1}{2},1)$,  the derived error bound is suboptimal with respect to regularity. To conclude, we notice that the error bounds of \cite[Theorem 5.2]{MR4358465} hold under the assumption that $\partial \Omega \in C^{\infty}$. We improve upon them by assuming that $\Omega$ is a Lipschitz polytope that additionally satisfies an exterior ball condition when $s < \tfrac{1}{2}$.
\end{remark}

We now present error estimates for $s \in (0,\tfrac{1}{2})$ that are suboptimal in terms of regularity. When $s \in [\tfrac{1}{4},\tfrac{1}{2})$ the derived error estimates \EO{hold} under weaker regularity assumptions that the ones stated in Theorem \ref{thm:error_estimates_state_equation}.

\begin{theorem}[error estimates]
Let $n \geq 2$, $s \in (0,\tfrac{1}{2})$, and $r>n/2s$. Let $\Omega \subset \mathbb{R}^n$ be a bounded Lipschitz domain. Assume that $a$ is as in the statement of Theorem \ref{thm:stata_equation_well_posedness_integral} and satisfies, in addition, \ref{eq:assumption_on_a_phi}. If, in addition, $a = a(x,u)$ is locally Lipschitz in $u$, uniformly for $x \in \Omega$, $a(\cdot,0) \in L^{2}(\Omega)$ and $f \in L^2(\Omega) \cap L^r(\Omega)$, then we have the following a priori error estimate in \EO{the} energy-norm:
\begin{equation}
\| u - u_{\T} \|_{s}  \lesssim h_{\T}^{s} | \log h_{\T} |^{\frac{1}{2}+\zeta} \| f - a(\cdot,0) \|_{L^2(\Omega)},
\qquad
s \in (0,\tfrac{1}{2})\EO{.}
\label{eq:error_estimate_semilinear_s_final_1412_2}
\end{equation} 
Here, $\zeta$ is as in the statement of Proposition \ref{pro:state_regularity_Lipschitz_new}. If, in addition, \ref{eq:assumption_on_a_phi} holds with \EO{$\mathfrak{r}$ replaced by $\mathfrak{v} = n/s$}, then we have the following a priori error estimate in $L^2(\Omega)$:
\begin{equation}
\| u - u_{\T} \|_{L^2(\Omega)} \lesssim h_{\T}^{2s}| \log h_{\T} |^{2\left(\frac{1}{2}+\zeta\right)} \| f - a(\cdot,0) \|_{L^2(\Omega)},
\qquad
s \in (0,\tfrac{1}{2}).
\label{eq:error_estimate_semilinear_s_final_L2_1412_2}
\end{equation}
In both estimates the hidden constant is independent of $u$, $u_{\T}$, and $h_{\T}$.
\label{thm:error_estimates_state_equation_2}
\end{theorem}
\begin{proof}
The proof follows the arguments elaborated in the proof of Theorem \ref{thm:error_estimates_state_equation} but now utilizing the regularity results of Theorem \ref{thm:regularity_space_state_equation_integral_SobolevII}. 
\end{proof}

\subsubsection{Error estimates on suitable graded meshes}
Let us assume that we have at hand a family of meshes $\{ \T \}$ of $\bar \Omega$ such that, in addition to shape regularity, $\{ \T \}$ satisfies a suitable mesh refinement near the boundary of $\Omega$ \cite{MR3893441,MR4283703}: Given a mesh parameter $h>0$,  there is a number $\mu \geq 1$ such that for every $T \in \T$
\begin{equation}
h_T \leq C(\sigma) h^{\mu}
\textrm{ if }
T \cap \partial \Omega \neq \emptyset,
\quad
h_T \leq C(\sigma) h \mathrm{dist}(T,\partial \Omega)^{(\mu-1)/\mu}
\textrm{ if }
T \cap \partial \Omega = \emptyset.
\label{eq:graded_meshes_state_equation}
\end{equation}
Here, $C(\sigma)$ denotes a constant that only depends on the shape regularity coefficient $\sigma$ of  $\{ \T \}$. The number of degrees of freedom $\mathscr{N}$ of the corresponding finite element space $\mathbb{V}(\T)$ can be related to the discretization parameter $h$ as follows \cite[(3.13)]{MR4283703}:
\[
\mathscr{N} \approx h^{-n} 
\textrm{ if }
\mu < \tfrac{n}{n-1},
\quad
\mathscr{N} \approx h^{-n} |\log h|
\textrm{ if }
\mu = \tfrac{n}{n-1},
\quad
\mathscr{N} \approx h^{(1-n)\mu} 
\textrm{ if }
\mu > \tfrac{n}{n-1}.
\]
If $\mu \leq n/(n-1)$, $h$ and $\mathscr{N}$ satisfy the optimal relation $h \approx N^{-\frac{1}{n}}$ (up to a logarithmic factor if $\mu = n/(n-1)$.)

We now present error estimates on the graded meshes dictated by \eqref{eq:graded_meshes_state_equation} that improve the ones obtained in Theorem \ref{thm:error_estimates_state_equation} for shape regular families of conforming and simplicial triangulations. Since we will utilize the results of Theorem \ref{thm:regularity_space_state_equation_integral_Holder}, we will assume that $\Omega$ satisfies an exterior ball condition.

\begin{theorem}[error estimates on graded meshes]
Let $n \geq 2$, $s \in (0,1)$, and let $\Omega \subset \mathbb{R}^n$ be a bounded Lipschitz domain satisfying an exterior ball condition. Let $\mu = n/(n-1)$ be the parameter that dictates the mesh refinement \eqref{eq:graded_meshes_state_equation}, and let $\beta_{\star} = n/(2(n-1)) - s$. Assume that $a$ is as in the statement of Theorem \ref{thm:stata_equation_well_posedness_integral}.
If, in addition, $s \geq n/(4(n-1))$, 
$a(\cdot,0) \in L^{\infty}(\Omega)$, 
$a$ satisfies \eqref{eq:Nemitskii} with $\varrho = \beta_{\star}$, and $f \in C^{\beta}(\bar \Omega)$, where $\beta \geq \beta_{\star}$, then we have the following error \EO{bounds}
\begin{align}
\| u - u_{\T} \|_{s}  
& 
\lesssim h^{\tfrac{n}{2(n-1)}} | \EO{\log h} |^{\upsilon},
\quad
s \in \left[ \tfrac{n}{4(n-1)},1\right),
\label{eq:graded_estimate_s}
\\
\| u - u_{\T} \|_{L^2(\Omega)} 
&
\lesssim h^{\tfrac{n}{2(n-1)}+\vartheta}| \EO{\log h} |^{\wp},
\quad
s \in \left[ \tfrac{n}{4(n-1)},1\right),
\quad
\vartheta = \min \{s,\tfrac{1}{2}\}.
\label{eq:graded_estimate_2}
\end{align} 
where $\upsilon = 1$ if $s \neq \frac{1}{2}$ and $\upsilon = 2$ if $s = \frac{1}{2}$, $\EO{\wp} = \tfrac{3}{2}$ if $s > \tfrac{1}{2}$, $\EO{\wp} = \tfrac{5}{2} + \zeta$ if $s = \tfrac{1}{2}$, and $\EO{\wp} = \tfrac{3}{2} + \zeta$ if $s < \tfrac{1}{2}$.
The constant $\zeta$ is as in the statement of Proposition \ref{pro:state_regularity_Lipschitz_new}. In both estimates the hidden constant is independent of $u$, $u_{\T}$, and $h_{\T}$.
\label{thm:error_estimates_state_equation_graded}
\end{theorem}
\begin{proof}
We begin the proof by noticing that, since $a$ satisfies \eqref{eq:Nemitskii} with $\varrho = \beta_{\star}$, then \ref{eq:assumption_on_a_phi} holds. We can thus invoke the best approximation result \eqref{eq:error_estimate_semilinear_s} of Theorem \ref{thm:error_estimates_state_equation} to  immediately deduce the error bound
$
\| u - u_{\T} \|_{s} \lesssim  \| u - \Pi_{\T} u\|_{s},
$
where $\Pi_{\T}$ denotes the Scott--Zhang operator. The desired error estimate \eqref{eq:graded_estimate_s} is thus a consequence of \cite[estimate (3.14)]{MR4283703} and \cite[Theorem 3.5]{MR4283703} upon obtaining that $f - a(\cdot,u) \in C^{\beta_{\star}}(\bar \Omega)$. To accomplish this task, we invoke the regularity results of Theorem \ref{thm:regularity_space_state_equation_integral_Holder} to deduce that $u \in C^s(\mathbb{R}^n)$. Since $a$ satisfies \eqref{eq:Nemitskii} with $\varrho = \beta_{\star}$, the arguments in Remark \ref{rk:Nemitskii} reveal that $a(\cdot,u) \in C^{\beta_{\star}}(\bar \Omega)$. Notice that $s \geq \beta_{\star} = n/(2(n-1)) -s$ because $s \geq n/(4(n-1))$. Consequently, $f - a(\cdot,u) \in C^{\gamma}(\bar \Omega)$, where $\gamma = \min \{ \beta, \beta_{\star} \} = \beta_{\star}$.

With the previous regularity result at hand, the error estimate \eqref{eq:graded_estimate_2} in $L^2(\Omega)$ follows from \cite[Proposition 3.10]{MR4283703}. This concludes the proof.
\end{proof}

\subsection{A convergence property}
\EO{Let $\Omega$ be a bounded Lipschitz domain. Assume that $a$ is as in the statement of Theorem \ref{thm:stata_equation_well_posedness_integral} and satisfies, in addition, \ref{eq:assumption_on_a_phi}. Let $u \in \tilde H^s(\Omega)$ be the solution to \eqref{eq:weak_semilinear_pde_integral}, and let $\mathfrak{u}_{\T}  \in \mathbb{V}(\T)$ be the solution to \eqref{eq:discrete_semilinear_pde} with $f$ replaced by $f_{\T} \in L^{r}(\Omega)$, with $r>n/2s$. Then, we have \cite[Proposition 5.3]{MR4358465}}
\begin{equation}
\EO{f_{\T} \rightharpoonup f \textrm{ in }L^{r}(\Omega)
\implies
\mathfrak{u}_{\T} \rightarrow u \textrm{ in } L^{q}(\Omega),
\qquad
h_{\T} \downarrow 0,
\qquad
q \leq 2n/(n-2s).}
\label{eq:convergence property}
\end{equation}

\section{Finite element approximation of the adjoint equation}
\label{sec:fem_adjoint}
We introduce the following finite element approximation of \eqref{eq:adj_eq_integral}: Find $q_{\T} \in \V(\T)$ such that
\begin{equation}
 \label{eq:discrete_adjoint}
  \mathcal{A} (v_{\T},q_{\T})  + \left( \frac{\partial a}{\partial u}(\cdot,u) q_{\T}, v_{\T} \right)_{L^2(\Omega)} 
  = 
  \left(\frac{\partial L}{\partial u} (\cdot,u), v_{\T} \right)_{L^2(\Omega)} \quad \forall v_{\T} \in \V(\T),
\end{equation}
where $u \in \tilde H^s(\Omega) \cap L^{\infty}(\Omega)$ corresponds to the solution to \eqref{eq:weak_st_eq_integral}. We observe  that assumption \textnormal{\ref{B2}} guarantees that  $\partial L/ \partial u(\cdot,u) \in L^r(\Omega)$ for $r>n/2s$ while assumption \textnormal{\ref{A2}} reveals that  $\partial a/\partial u (x,u) \geq 0$ for a.e.~$x \in \Omega$ and for all $u \in \mathbb{R}$. The existence and uniqueness of a discrete solution $q_{\T} \in \mathbb{V}(\T)$ to problem \eqref{eq:discrete_adjoint} is thus immediate.

Since \EO{it will be useful, we present the following Galerkin orthogonality property:}
\begin{equation}
\mathcal{A}(v_{\T},p - q_{\T}) + \left( \frac{\partial a}{\partial u}(\cdot,u) (p - q_{\T}), v_{\T} \right)_{L^2(\Omega)} = 0 
\qquad
\forall v_{\T} \in \mathbb{V}(\T). 
\label{eq:Galerkin_adjoint}
\end{equation}

To \EO{present error estimates concisely, we define}
\[
 \EO{\Upsilon(L):= \left\| \frac{\partial L}{\partial u} (\cdot,\bar u) \right\|_{L^2( \Omega)},
 \quad
 \Psi(L,a):=\left\| \frac{\partial L}{\partial u} (\cdot,\bar u) \right\|_{C^{\frac{1}{2}-s}(\bar\Omega)}\left[ 1 + \| \bar{z} - a(\cdot,0) \|_{L^{\infty}(\Omega)} \right].}
\]

\begin{theorem}[error estimates]
\EO{Let $n \geq 2$, $s \in [\tfrac{1}{4},1)$, and let $\Omega$ be a bounded Lipschitz domain} such that it satisfies an exterior ball condition for $s<1/2$. Assume that \ref{A1}--\ref{A3}, \ref{B1}--\ref{B2}, \ref{C1}, and \ref{D1}--\ref{D2} hold.  Let $p \in \tilde H^s(\Omega)$ be the solution to \eqref{eq:adj_eq_integral}, and let $q_{\T} \in \mathbb{V}(\T)$ be the solution to the discrete problem \eqref{eq:discrete_adjoint}. Then, we have the following a priori error estimates in \EO{the} energy-norm:
\begin{align}
\| p - q_{\T} \|_{s} & \lesssim h_{\T}^{\frac{1}{2} } | \log h_{\T} |\EO{\Psi(L,a)},
\qquad
s \in [\tfrac{1}{4},\tfrac{1}{2}),
\label{eq:estimate_adjoint_equation_s_final_1412}
\\
\| p - q_{\T} \|_{s} & \lesssim  h_{\T}^{\frac{1}{2}}|\log h_{\T}|^{\frac{3}{2} + \zeta}\EO{\Upsilon(L)},
\qquad
s = \tfrac{1}{2},
\label{eq:estimate_adjoint_equation_s_final_12}
\\
\| p - q_{\T} \|_{s} & \lesssim h_{\T}^{\frac{1}{2} } | \log h_{\T} |^{\frac{1}{2}}\EO{\Upsilon(L)},
\qquad
s \in (\tfrac{1}{2},1).
\label{eq:estimate_adjoint_equation_s_final_121}
\end{align} 
Here, $\zeta$ is as in the statement of Proposition \ref{pro:state_regularity_Lipschitz_new}. \EO{Let $\vartheta = \min \{ s , \tfrac{1}{2}\}$.} We also have the following a priori error estimates in $L^2(\Omega)$:
\begin{align}
\| p - q_{\T} \|_{L^2(\Omega)} & \lesssim h_{\T}^{\vartheta + \frac{1}{2}}| \log h_{\T} |^{\frac{3}{2}+\zeta}\EO{\Psi(L,a)},
\qquad
s \in [\tfrac{1}{4},\tfrac{1}{2}).
\label{eq:estimate_adjoint_equation_s_final_L2_1412}
\\
\| p - q_{\T} \|_{L^2(\Omega)} & \lesssim h_{\T}^{\vartheta + \frac{1}{2}}| \log h_{\T} |^{2\left(\frac{3}{2} + \zeta\right)}\EO{\Upsilon(L)},
\qquad
s = \tfrac{1}{2}.
\label{eq:estimate_adjoint_equation_s_final_L2_12}
\\
\| p - q_{\T} \|_{L^2(\Omega)} & \lesssim h_{\T}^{\vartheta + \frac{1}{2}}| \log h_{\T} |\EO{\Upsilon(L)},
\qquad
s \in (\tfrac{1}{2},1).
\label{eq:estimate_adjoint_equation_s_final_L2_121}
\end{align}
In all estimates, the hidden constant is independent of  $p$, $q_h$, and $h_{\T}$.
\label{thm:error_estimates_adjoint_equation}
\end{theorem}
\begin{proof}
\EO{In view of $\| p - q_{\T} \|_{s}^2 = \mathcal{A}(p - q_{\T},p - q_{\T})$, setting $v_{\T}$ as $q_{\T}$ in \eqref{eq:Galerkin_adjoint} yields
\begin{equation*}
\| p - q_{\T} \|_{s}^2 
=
 \mathcal{A}(p - q_{\T},p ) + \left(  \frac{\partial a}{\partial u}(\cdot,u) (p -q_{\T} ), q_{\T} \right)_{L^2(\Omega)}.
\end{equation*}
We now utilize Galerkin orthogonality again, but now as it is in \eqref{eq:Galerkin_adjoint}, to deduce}
\begin{equation*}
\| p - q_{\T} \|_{s}^2 
=
   \mathcal{A}(p - q_{\T},p -v_{\T}) +   \left(  \frac{\partial a}{\partial u}(\cdot,u) (p-q_{\T} ) ,q_{\T} - v_{\T} \right)_{L^2(\Omega)}
  \, \forall v_{\T} \in \mathbb{V}(\T).
\end{equation*}
Write $q_{\T} - v_{\T} = (q_{\T} - p) + (p-v_{\T})$, observe that $(\partial a/\partial u(\cdot,u) (p - q_{\T}), q_{\T} - p )_{L^2(\Omega)} \leq 0$, and utilize assumption \ref{A3} combined with the fact that $u \in L^{\infty}(\Omega)$ to obtain
\[
\| p - q_{\T} \|_{s}^2 \leq\| p - q_{\T} \|_{s} \| p - v_{\T} \|_{s} + C_{\mathfrak{m}}\| p - q_{\T} \|_{L^2(\Omega)} \| p - v_{\T} \|_{L^2(\Omega)},
\quad v_{\T} \in \mathbb{V}(\T).
\]
This bound allows us to obtain the quasi-best approximation property: $\| p - q_{\T} \|_{s} \lesssim \inf\{ \| p - v_{\T} \|_{s}: v_{\T} \in \V(\T)\}$. The energy-norm error estimates \eqref{eq:estimate_adjoint_equation_s_final_1412}--\eqref{eq:estimate_adjoint_equation_s_final_121} thus follow from similar arguments to the ones developed in the proof of Theorem \ref{thm:error_estimates_state_equation} upon utilizing the regularity estimates \eqref{eq:reg_p_s121}, \eqref{eq:reg_p_s12}, and \eqref{eq:reg_p_1412}.
The $L^2(\Omega)$-norm error bounds \eqref{eq:estimate_adjoint_equation_s_final_L2_1412}--\eqref{eq:estimate_adjoint_equation_s_final_L2_121} follow from a duality argument. This concludes the proof.
\end{proof}

Let us now introduce $u_{\T} \in \mathbb{V}(\T)$ as the solution to \eqref{eq:discrete_semilinear_pde} with $f$ replaced by $z_{\T}$; $z_{\T}$ corresponds to an arbitrary piecewise constant function over the mesh $\T$. We also introduce the discrete function $p_{\T} \in \mathbb{V}(\T)$ as the solution to
\begin{equation}
 \label{eq:discrete_adjoint_2}
  \mathcal{A} (v_{\T}, p_{\T})  + \left( \frac{\partial a}{\partial u}(\cdot,u_{\T}) p_{\T}, v_{\T} \right)_{L^2(\Omega)} 
  = 
  \left(\frac{\partial L}{\partial u} (\cdot,u_{\T}), v_{\T} \right)_{L^2(\Omega)}
\end{equation}
for all $v_{\T} \in \V(\T)$. In what follows, we analyze error bounds for $p - p_{\T}$. To accomplish this task, we define $q$ as the solution to the weak problem: Find $q \in \tilde H^s(\Omega)$ such that
\begin{equation}
\label{eq:q}
  \mathcal{A} (v,q)  + \left( \frac{\partial a}{\partial u}(\cdot,u_{\T}) q, v \right)_{L^2(\Omega)} 
  = 
 \left(\frac{\partial L}{\partial u} (\cdot,u_{\T}), v \right)_{L^2(\Omega)} \quad \forall v \in \tilde H^s(\Omega).
\end{equation}

Since we are operating under \emph{local} assumptions on $a= a(x,u)$ and $L=L(x,u)$, i.e., assumptions that hold for $u$ on bounded intervals of $\mathbb{R}$, in what follows we assume that solutions $u_{\T}$ to problem \eqref{eq:discrete_semilinear_pde} are uniformly bounded in $L^{\infty}(\Omega)$, i.e.,
 \begin{equation}
 \exists C >0: \quad \| u_{\T} \|_{L^{\infty}(\Omega)} \leq C \quad \forall \T \in \mathbb{T}.
 \label{eq:u_h_bounded}
 \end{equation}
 
With \eqref{eq:u_h_bounded} at hand, the assumptions imposed on the data allow us to conclude that \eqref{eq:discrete_adjoint_2} and \eqref{eq:q} are well-posed. In particular, there exists a unique solution $q \in \tilde H^s(\Omega) \cap L^{\infty}(\Omega)$ to \eqref{eq:q}. If, for every $\mathfrak{m}>0$ and $u \in [-\mathfrak{m},\mathfrak{m}]$, $\partial L/\partial u(\cdot,u) \in L^2(\Omega)$, we can apply Proposition \ref{pro:state_regularity_Lipschitz_new} to deduce that 
$
q \in H^{s+\theta - \epsilon}(\Omega)
$
together with  
\begin{equation}
\begin{aligned}
   \| q \|_{H^{2s  -2\epsilon}(\Omega)} 
  &
  \lesssim 
  \epsilon^{-\frac{1}{2}-\zeta} \left\| \frac{\partial L}{\partial u} (\cdot,u_{\T}) \right\|_{L^2(\Omega)}, \quad s \in (0, \tfrac{1}{2}],
   \quad \forall 0<\epsilon<s,
   \\
  \| q \|_{H^{s + \frac{1}{2}-\epsilon}(\Omega)} 
  &
  \lesssim 
  \epsilon^{-\frac{1}{2}} \left\| \frac{\partial L}{\partial u} (\cdot,u_{\T}) \right\|_{L^2(\Omega)}, \quad s \in (\tfrac{1}{2},1),
   \quad \forall 0<\epsilon<s+\tfrac{1}{2}.
  \end{aligned}
 \label{eq:regularity_state_Lipschitz_new_p}
  \end{equation}
\EO{With \eqref{eq:regularity_state_Lipschitz_new_p} at hand, the arguments elaborated in the proof of Theorem \ref{thm:error_estimates_adjoint_equation} yield}
\begin{equation}
\| q - p_{\T} \|_s \lesssim h_{\T}^{\vartheta}|\log h_{\T}|^{\upsilon} \left \| \frac{\partial L}{\partial u} (\cdot,u_{\T}) \right \|_{L^2(\Omega)},
\quad
0<s<1,
\quad
\vartheta = \min \{ s , \tfrac{1}{2} \}.
\label{eq:q-ph}
\end{equation}
Here, $\upsilon = \tfrac{1}{2}$ if $s > \tfrac{1}{2}$, $\upsilon = \tfrac{3}{2} + \zeta$ if $s = \tfrac{1}{2}$, and $\upsilon = \tfrac{1}{2} + \zeta$ if $s < \tfrac{1}{2}$. $\zeta$ is as in the statement of Proposition \ref{pro:state_regularity_Lipschitz_new}. In addition, we have an error bound in $L^2(\Omega)$:
\begin{equation}
\| q - p_{\T} \|_{L^2(\Omega)} \lesssim h_{\T}^{2\vartheta} |\log h_{\T}|^{2\upsilon} \left \| \frac{\partial L}{\partial u} (\cdot,u_{\T}) \right\|_{L^2(\Omega)},
\quad
0<s<1.
\label{eq:q-ph_L2}
\end{equation}

A \EO{second ingredient within the analysis of error bounds for $p - p_{\T}$ is to introduce} 
\begin{equation}
\label{eq:mathsf_u}
y \in \tilde H^s(\Omega):
\quad
  \mathcal{A}(y,v)  +  \langle a(\cdot,y),v \rangle = \langle z_{\T} , v \rangle 
  \quad \forall v \in \tilde H^{s}(\Omega).
\end{equation}
The well-posedness of \eqref{eq:mathsf_u} follows from Theorem \ref{thm:stata_equation_well_posedness_integral}; observe that $z_{\T} \in L^{\infty}(\Omega)$ for every $\T \in \mathbb{T}$. In particular, we have that $y \in \tilde H^s(\Omega) \cap L^{\infty}(\Omega)$. If, for every $\mathfrak{m}>0$ and $u \in [-\mathfrak{m},\mathfrak{m}]$, $a(\cdot,u) \in L^2(\Omega)$, the fact that $z_{\T} \in L^{2}(\Omega)$, uniformly with respect to discretization, allows us to conclude the following regularity result: $y \in H^{s+\theta-\epsilon}(\Omega)$, where $\theta = \tfrac{1}{2}$ for $\tfrac{1}{2} < s < 1$ and $\theta = s -\epsilon > 0$ for $0 < s \leq \tfrac{1}{2}$.

We \EO{are now in position to derive error estimates.}
\begin{theorem}[error estimates]
Let $n \geq 2$ and $s \in (0,1)$. Let $\Omega$ be a bounded Lipschitz domain. Assume that \ref{A1}--\ref{A3}, \ref{B1}--\ref{B2},  
and \eqref{eq:u_h_bounded} hold. Assume, in addition, that, for every $\mathfrak{m}>0$ and $u \in [-\mathfrak{m},\mathfrak{m}]$,
\begin{equation}
 a(\cdot,u), 
 \,
 \frac{\partial L}{\partial u}(\cdot,u) \in L^2(\Omega)
 \label{eq:assumptions_on_a_and_L_in_L2}
\end{equation}
and that \EO{$\partial L/\partial u = \partial L/\partial u(x,u)$ is locally Lipschitz with respect to $u$.}
Let $p \in \tilde H^s(\Omega)$ be the solution to \eqref{eq:adj_eq_integral}, and let $p_{\T} \in \mathbb{V}(\T)$ be the solution to \eqref{eq:discrete_adjoint_2}. Then, we have 
\begin{equation}
\| p - p_{\T} \|_{s} \lesssim h_{\T}^{\vartheta}|\log h_{\T}|^{\upsilon} + \| z - z_{\T} \|_{L^2(\Omega)},
\qquad
\vartheta = \min \{ s , \tfrac{1}{2} \}.
\label{eq:estimate_adjoint_equation_s_2}
\end{equation}
If, in addition, $a$ satisfies \ref{eq:assumption_on_a_phi} \EO{with $\mathfrak{r}$ replaced by $\mathfrak{v} = n/s$,} then
\begin{equation}
\| p - p_{\T} \|_{L^2(\Omega)} \lesssim h_{\T}^{2\vartheta}|\log h_{\T}|^{2\upsilon} + \| z - z_{\T} \|_{L^2(\Omega)},
\qquad
\vartheta = \min \{ s , \tfrac{1}{2} \}.
\label{eq:estimate_adjoint_equation_s_3}
\end{equation}
Here, $\upsilon = \tfrac{1}{2}$ if $s > \tfrac{1}{2}$, $\upsilon = \tfrac{3}{2} + \zeta$ if $s = \tfrac{1}{2}$, and $\upsilon = \tfrac{1}{2} + \zeta$ if $s < \tfrac{1}{2}$. $\zeta$ is as in Proposition \ref{pro:state_regularity_Lipschitz_new}. In both estimates, the hidden constant is independent of $h_{\T}$.
\label{thm:error_estimates_adjoint_equation_2}
\end{theorem}
\begin{proof}
We follow the proof of \cite[Theorem 6.2]{MR4358465} and bound $\| p - p_{\T} \|_{s}$ as follows: $\| p - p_{\T} \|_{s} \leq \| p - q \|_{s} + \| q - p_{\T} \|_{s}$, where $q$ denotes the solution to \eqref{eq:q}. \EO{The control of $\| q - p_{\T} \|_{s}$ follows from \eqref{eq:q-ph}: 
$\| q - p_{\T} \|_{s} \lesssim h_{\T}^{\vartheta}|\log h_{\T}|^{\upsilon}$.}
It thus suffices to bound $\| p - q \|_{s}$. To accomplish this task, let us first observe that, for every $v \in \tilde H^s(\Omega)$,
\begin{multline*}
p- q \in \tilde H^s(\Omega):
\quad
  \mathcal{A} (v,p-q)  + \left( \frac{\partial a}{\partial u}(\cdot,u) (p-q), v \right)_{L^2(\Omega)} 
 \\
  = 
   \left( \left[ \frac{\partial a}{\partial u}(\cdot,u_{\T}) -  \frac{\partial a}{\partial u}(\cdot,u) \right]q, v \right)_{L^2(\Omega)} 
   +
  \left( \frac{\partial L}{\partial u} (\cdot,u) - \frac{\partial L}{\partial u} (\cdot,u_{\T}), v \right)_{L^2(\Omega)}.
\end{multline*}
Set $v = p-q \in \tilde H^s(\Omega)$ and use that \EO{$\partial a/\partial u = \partial a/\partial u(x,u)$ and $ \partial L/\partial u = \partial L/\partial u(x,u)$} are locally Lipschitz in $u$ to obtain
$
\| p - q \|_s \lesssim \| u - u_{\T} \|_{L^2(\Omega)} [ 1 + \| q \|_{L^{\infty}(\Omega)}].
$

We \EO{now bound $\| u - u_{\T} \|_{L^2(\Omega)}$ on the basis of similar arguments}: $ \| u - u_{\T} \|_{L^2(\Omega)} \leq \| u - y \|_{L^2(\Omega)} + \| y - u_{\T} \|_{L^2(\Omega)}$, where $y$ denotes the solution to \eqref{eq:mathsf_u}. Since $z_{\T} \in L^2(\Omega)$, uniformly with respect to discretization, and \eqref{eq:assumptions_on_a_and_L_in_L2} holds, we have at hand the regularity estimates \eqref{eq:regularity_state_Lipschitz_new} for $y$. \EO{These estimates and 
the arguments elaborated in the proofs of Theorems \ref{thm:error_estimates_state_equation} and \ref{thm:error_estimates_adjoint_equation} yield $\| y - u_{\T} \|_{L^2(\Omega)} \lesssim h_{\T}^{2\vartheta}|\log h_{\T}|^{2\upsilon}$.}
To bound $\| u - y \|_{L^2(\Omega)}$, we write the problem that 
$u - y $ solves and derive a stability estimate on the basis of \ref{A1}--\ref{A3}:
$ \| u -y \|_{L^2(\Omega)} \lesssim \| z - z_{\T} \|_{L^2(\Omega)}$. A collection of the derived estimates yield \eqref{eq:estimate_adjoint_equation_s_2}. The proof of \eqref{eq:estimate_adjoint_equation_s_3} follows similar arguments.
\end{proof}

\section{Finite element approximation of the fractional control problem}
\label{sec:fem_control}

We consider two strategies to discretize the optimal control problem \eqref{eq:min_integral}--\eqref{eq:weak_st_eq_integral}: a semidiscrete approach where \emph{the admisible control set is not discretized} and a fully discrete strategy where control variables are discretized with piecewise constant functions.

\subsection{A fully discrete scheme}
\label{sec:discrete_optimal_control_problem}
We consider the following fully discrete approximation of the PDE-constrained optimization problem \eqref{eq:min_integral}--\eqref{eq:weak_st_eq_integral}: Find
\begin{equation}\label{eq:min_discrete}
\min \{ J(u_{\T},z_{\T}): (u_{\T},z_{\T}) \in \mathbb{V}(\T) \times \mathbb{Z}_{ad}(\T) \}
\end{equation}
subject to the \emph{discrete state equation}
\begin{equation}\label{eq:weak_st_eq_discrete}
\mathcal{A}( u_{\T}, v_{\T})+\int_{\Omega} a(x,u_{\T}(x)) v_{\T}(x) \mathrm{d}x = \int_{\Omega} z_{\T}(x) v_{\T}(x) \mathrm{d}x \quad \forall v_{\T} \in \mathbb{V}(\T).
\end{equation}
Here, 
$
 \mathbb{Z}_{\mathrm{ad}}(\T) = \mathbb{Z}_{\mathrm{ad}} \cap  \mathbb{Z}(\T)
$
and 
$
 \mathbb{Z}(\T) = \left\{ v_{\T} \in L^{\infty}( \Omega ): v_{\T}|_T \in \mathbb{P}_0(T) \ \forall T \in \T  \right\}.
$

The existence of a solution follows standard arguments. \EO{To present first order optimality conditions, we introduce} $\mathcal{S}_{\T}: \mathbb{Z}(\T) \ni z_{\T} \mapsto u_{\T} \in \mathbb{V}(\T)$ and $j_{\T}(z_{\T}):= J(\mathcal{S}_{\T} z_{\T},z_{\T})$. If $\bar{z}_{\T}$ is a local minimum for \eqref{eq:min_discrete}--\eqref{eq:weak_st_eq_discrete}, then $(\bar u_{\T}, \bar p_{\T}, \bar z_{\T}) \in \mathbb{V}(\T) \times \mathbb{V}(\T) \times \mathbb{Z}_{ad}(\T)$ satisfies the optimality system
\begin{align}
  \mathcal{A} ( \bar u_{\T},v_{\T})  +  ( a(\cdot, \bar u_{\T}), v_{\T} )_{L^2(\Omega)} & = (\bar z_{\T}, v_{\T})_{L^2(\Omega)}  
 \label{eq:optimal_state_discrete}
\\
\mathcal{A}(v_{\T},\bar p_{\T}) +  \left(  \frac{\partial a}{\partial u}(\cdot, \bar u_{\T}) \bar p_{\T}, v_{\T} \right)_{L^2(\Omega)}  & = \left( \frac{\partial L}{\partial u}(\cdot,\bar u_{\T}), v_{\T} \right)_{L^2(\Omega)}
 \label{eq:optimal_adjoint_state_discrete}
 \\
 (\bar p_{\T} + \alpha \bar z_{\T}, z_{\T} - \bar z_{\T})_{L^2(\Omega)} & \geq 0 
 \quad
 \label{eq:variational_inequality_discrete}
\end{align}
\EO{for all $ v_{\T} \in \mathbb{V}({\T})$ and for all $z_{\T} \in \mathbb{Z}_{ad}(\T)$.}

\subsubsection{Convergence of discretizations}
\label{sec:convergence}
\EO{The following result improves upon \cite[Theorem 7.2]{MR4358465}: it requires weaker assumptions on $\Omega$ and $L$.}
\begin{theorem}[convergence]
\EO{Let $n \geq 2$ and $s \in (0,1)$. Let $\Omega$ be a bounded Lipschitz domain. Assume that \ref{A1}--\ref{A3} and \ref{B1}--\ref{B2} hold. Assume, in addition, that
\eqref{eq:u_h_bounded} hold. Let $\bar z_{\T}$ be a global solution of \eqref{eq:min_discrete}--\eqref{eq:weak_st_eq_discrete} for $\T \in \mathbb{T}$. Then, there exist nonrelabeled subsequences $\{ \bar z_{\T} \}$ such that $\bar z_{\T} \mathrel{\ensurestackMath{\stackon[1pt]{\rightharpoonup}{\scriptstyle\ast}}} \bar{z}$ in $L^{\infty}(\Omega)$ as $h_{\T} \downarrow 0$, with $\bar z$ being a global solution of \eqref{eq:min_integral}--\eqref{eq:weak_st_eq_integral}. In addition, we have}
\begin{equation}
\label{eq:convergence}
\| \bar z - \bar z_{\T} \|_{L^{2}(\Omega)} \rightarrow 0,
\qquad
j_{\T}( \bar z_{\T}) \rightarrow j(\bar z), 
\end{equation}
as $h_{\T} \downarrow 0$.
\label{thm:convergence}
\end{theorem}
\begin{proof}
Since $\{ \bar z_{\T} \}$ is uniformly bounded in $L^{\infty}(\Omega)$, we deduce the existence of a nonrelabeled subsequence $\{ \bar z_{\T} \}$ such that $\bar z_{\T} \mathrel{\ensurestackMath{\stackon[1pt]{\rightharpoonup}{\scriptstyle\ast}}} \bar{z}$ in $L^{\infty}(\Omega)$ as $h_{\T} \downarrow 0$. Let $\tilde z \in \mathbb{Z}_{ad}$ be a global solution of \eqref{eq:min_integral}--\eqref{eq:weak_st_eq_integral}. Define $\tilde z_{\T} \in \mathbb{Z}_{ad}(\T)$ by $\tilde z_{\T}|_{T} := \int_{T} \tilde z(x) \mathrm{d}x / |T|$ for $T \in \T$ and define $\tilde p$ as the solution to \eqref{eq:adj_eq_integral} with $u$ replaced by $\tilde u := \mathcal{S} \tilde z$. \EO{Observe that $\tilde p \in \tilde{H}^s(\Omega) \cap L^{\infty}(\Omega)$.}
Invoke the projection formula \eqref{eq:projection_control} and \cite[Theorem 1]{MR1173747} to obtain $\tilde z \in H^{\EO{s}}(\Omega)$. Consequently, $\| \tilde z - \tilde z_{\T} \|_{L^{2}(\Omega)} \rightarrow 0$ as $h_{\T} \downarrow 0$. The rest of the proof follows the arguments elaborated in the proof of \cite[Theorem 7.2]{MR4358465}.
\end{proof}

We \EO{now present a second convergence result \cite[Theorem 7.3]{MR4358465}.}

\begin{theorem}[convergence]
Let the assumptions of Theorem \ref{thm:convergence} hold.
Let $\bar z$ be a strict local minimum of problem \eqref{eq:min_integral}--\eqref{eq:weak_st_eq_integral}. Then, there exists a sequence $\{ \bar z_{\T} \}$ of local minima of the \EO{fully} discrete optimal control problems \EO{satisfying \eqref{eq:convergence}.}
\label{thm:convergence_local_minima}
\end{theorem}

\subsubsection{Error estimates}
\label{sec:error_estimates}
Let $\{ \bar z_{\T} \} \EO{\subset} \mathbb{Z}_{ad}(\T)$ be a sequence of local minima of \EO{\eqref{eq:min_discrete}--\eqref{eq:weak_st_eq_discrete}}
such that $\| \bar z - \bar z_{\T} \|_{L^{2}(\Omega)} \rightarrow 0$ as $h_{\T} \downarrow 0$; $\bar z$ being a local solution of the continuous problem \eqref{eq:min_integral}--\eqref{eq:weak_st_eq_integral}; see Theorems \ref{thm:convergence} and \ref{thm:convergence_local_minima}. \EO{In this section, we derive the bound \eqref{eq:error_estimate_in L2} for the error
$
\| \bar z - \bar z_{\T} \|_{L^2(\Omega)}.
$
The following result is instrumental.} 
%
\begin{theorem}[instrumental error bound]
\EO{Let the assumptions of Theorem \ref{thm:convergence} hold. Assume that, for every $\mathfrak{m}>0$ and $u \in [-\mathfrak{m},\mathfrak{m}]$, \eqref{eq:assumptions_on_a_and_L_in_L2} holds.
Let $\bar z \in \mathbb{Z}_{ad}$ \EO{satisfy} the conditions \eqref{eq:second_order_equivalent}. If \eqref{eq:error_estimate_in L2} is false, then there exists $h_{\star}>0$ such that}
\begin{equation}
\label{eq:basic_estimate}
\mathfrak{C}\| \bar z - \bar z_{\T} \|^2_{L^2(\Omega)} \leq \left[ j'(\bar z_{\T}) - j'(\bar z) \right](\bar z_{\T} - \bar z)
\quad
\forall
h_{\T} \leq h_{\star},
\quad
\mathfrak{C} = 2^{-1}\min \{ \EO{\nu}, \alpha \},
\end{equation}
where $\EO{\nu}$ is the constant appearing in \eqref{eq:second_order_equivalent} and $\alpha$ is the regularization parameter.
\label{thm:instrumental_error_estiamate}
\end{theorem}
\begin{proof}
\EO{With the equivalence \eqref{eq:second_order_equivalent} of Theorem \ref{thm:equivalent_opt_cond} at hand, the proof follows the arguments in \cite[Theorem 7.4]{MR4358465}.}
\end{proof}

We \EO{are now ready to derive a bound for the error $\bar z - \bar z_{\T}$ in $L^2(\Omega)$.}

\begin{theorem}[error estimate]
\EO{Let the assumptions of Theorem \ref{thm:instrumental_error_estiamate} hold. Assume, in addition, that $\partial L/ \partial u = \partial L/ \partial u(x,u)$ is locally Lipschitz in $u$. Let $\bar z \in \mathbb{Z}_{ad}$ satisfy the second order conditions \eqref{eq:second_order_equivalent}. Then, there exists $h_{\star} >0$ such that}
\begin{equation}
\label{eq:error_estimate_in L2}
\| \bar z - \bar z_{\T} \|_{L^2(\Omega)}\lesssim h_{\T}^{2\vartheta} |\log h_{\T}|^{2\upsilon}
\qquad
\forall h_{\T} \leq h_{\star},
\qquad
\vartheta = \min \{ s, \tfrac{1}{2} \}.
\end{equation} 
\EO{Here, $\upsilon = \tfrac{1}{2}$ if $s > \tfrac{1}{2}$, $\upsilon = \tfrac{3}{2} + \zeta$ if $s = \tfrac{1}{2}$, and $\upsilon = \tfrac{1}{2} + \zeta$ if $s < \tfrac{1}{2}$; $\zeta$ is as in the statement of Proposition \ref{pro:state_regularity_Lipschitz_new}. The hidden constant is independent of $h_{\T}$}
\label{thm:error_estimate_control}
\end{theorem}
\begin{proof}
We proceed by contradiction and assume that the desired error estimate \eqref{eq:error_estimate_in L2} does not hold so that we have at hand the instrumental one of Theorem \ref{thm:instrumental_error_estiamate}.

Set \EO{$z_{\T} = \Pi_{\T} \bar z$ in} \eqref{eq:variational_inequality_discrete} to deduce that $j_{\T}'(\bar z_{\T})(\Pi_{\T} \bar z - \bar z_{\T}) \geq 0$. Here, $\Pi_{\T}: L^2(\Omega) \rightarrow \mathbb{Z}(\T)$ denotes the orthogonal projection operator onto piecewise constant functions over $\T$. We now invoke the continuous variational inequality \eqref{eq:var_ineq_integral} to arrive at $j'(\bar z)(\bar z_{\T} - \bar z) \geq 0$. With these two inequalities at hand, we utilize \eqref{eq:basic_estimate} to obtain
\[
\mathfrak{C}
\| \bar z - \bar z_{\T} \|_{L^2(\Omega)}^2 \leq [j_{\T}'(\bar z_{\T}) -j'(\bar z_{\T})]( \Pi_{\T} \bar z - \bar z_{\T})  + j'(\bar z_{\T})( \Pi_{\T} \bar z - \bar z)  =: \mathrm{I}_{\T} + \mathrm{J}_{\T}.
\]

To \EO{bound the term $\mathrm{I}_{\T}$, we use the definition of $\Pi_{\T}$ and proceed as follows:}
\begin{equation*}
\begin{aligned}
 \mathrm{I}_{\T} & = (\bar p_{\T}  - p(\bar z_{\T}), \Pi_{\T} \bar z - \bar z_{\T})_{L^2(\Omega)} 
= 
(\bar p_{\T}  - p(\bar z_{\T}), \Pi_{\T} (\bar z - \bar z_{\T}) )_{L^2(\Omega)} 
\\
& \lesssim 
\| \bar p_{\T}  - p(\bar z_{\T}) \|_{L^2(\Omega)} \| \bar z - \bar z_{\T}\|_{L^2(\Omega)}
\leq 
C h_{\T}^{\EO{4\vartheta}}|\log h_{\T}|^{\EO{4\upsilon}}
+
\frac{\mathfrak{C}}{4} \| \bar z - \EO{\bar z_{\T}} \|^2_{L^2(\Omega)}.
\end{aligned}
\label{eq:IT}
\end{equation*}
Here, $p(\bar z_{\T})$ denotes the solution to \eqref{eq:adj_eq_integral} with $u$ replaced by $\mathcal{S} \bar z_{\T}$. \EO{The error bound $\| \bar p_{\T}  - p(\bar z_{\T}) \|_{L^2(\Omega)} \lesssim h_{\T}^{2\vartheta}|\log h_{\T}|^{2\upsilon}$ follows from the arguments elaborated within the proofs of Theorems \ref{thm:error_estimates_adjoint_equation} and \ref{thm:error_estimates_adjoint_equation_2}. For brevity, we skip the details.}

We control the term $\mathrm{J}_{\T}$ \EO{on the basis of similar} arguments. In fact, we have
\begin{equation*}
\begin{aligned}
\mathrm{J}_{\T} & = ( p(\bar z_{\T}) + \alpha \bar z_{\T}, \Pi_{\T} \bar z - \bar z)_{L^2(\Omega)} = ( p(\bar z_{\T}) , \Pi_{\T} \bar z - \bar z)_{L^2(\Omega)}
\\
& = (p(\bar z_{\T}) - \Pi_{\T} p(\bar z_{\T}),\Pi_{\T} \bar z - \bar z)_{L^2(\Omega)} \lesssim \EO{h_{\T}^{4\vartheta}| \log h_{\T} |^{2\eta}},
 \quad
 \vartheta = \min \{ s, \tfrac{1}{2} \},
\label{eq:IIT}
\end{aligned}
\end{equation*}
where \EO{$\eta = 0$ if $s>\tfrac{1}{2}$ and $\eta = \tfrac{1}{2}+\zeta$ if $s\leq\tfrac{1}{2}$. To simplify the presentation, we define $\mathfrak{e}_p := p(\bar z_{\T}) - \Pi_{\T} p(\bar z_{\T})$. To obtain the bound for $\textrm{J}_{\T}$, we have used the estimates
\begin{equation*}
 \begin{aligned}
& \| \mathfrak{e}_p \|_{L^2(\Omega)} \lesssim h_{\T},
\qquad
\| \Pi_{\T} \bar z - \bar z \|_{L^2(\Omega)} \lesssim h_{\T},
\quad
\textrm{ for } s > \tfrac{1}{2},
\\
&\| \mathfrak{e}_p \|_{L^2(\Omega)} \lesssim h^{2s}_{\T} |\log h_{\T}|^{\frac{1}{2}+\zeta},
\qquad
\| \Pi_{\T} \bar z - \bar z \|_{L^2(\Omega)} \lesssim h^{2s}_{\T} |\log h_{\T}|^{\frac{1}{2}+\zeta},
\quad
  \textrm{ for } s \leq \tfrac{1}{2}.
\end{aligned}
\end{equation*}
These bounds follow from standard error estimates for the orthogonal projection $\Pi_{\T}$ combined with the regularity results of Proposition \ref{pro:state_regularity_Lipschitz_new}, which guarantee that
\[
 p(\bar{z}_{\T}) \in H^1(\Omega) \textrm{ for } s>\tfrac{1}{2},
 \qquad
 p(\bar{z}_{\T}) \in H^{2s-2\epsilon}(\Omega) \textrm{ for } s\leq\tfrac{1}{2},
\]
together with the bound 
$
 \| p(\bar{z}_{\T}) \|_{H^{2s-2\epsilon}(\Omega)} \lesssim \epsilon^{-\frac{1}{2}-\zeta}.
$
Here, $\epsilon \in (0,s)$ and $\zeta$ is as in the statement of Proposition \ref{pro:state_regularity_Lipschitz_new}. Observe that $\partial L/ \partial u (\cdot, \mathcal{S} \bar z_{\T}) - \partial a/ \partial u (\cdot, \mathcal{S} \bar z_{\T})p(\bar z_{\T}) \in L^2(\Omega)$ uniformly with respect to discretization. We notice that the same regularity properties can be obtained for $\bar{z}$.}

Finally, \EO{we collect the bounds obtained for $\mathrm{I}_{\T}$ and $\mathrm{J}_{\T}$ to obtain $\| \bar z - \bar z_{\T} \|_{L^2(\Omega)} \lesssim h_{\T}^{2\vartheta}|\log h_{\T}|^{2\upsilon}$}. This is a contradiction and concludes the proof.
\end{proof}

\begin{remark}[improvements on the theory]
\rm
\EO{The error bound \eqref{eq:error_estimate_in L2} improves the one in  \cite[Theorem 7.5]{MR4358465} in several directions. First, \eqref{eq:error_estimate_in L2} holds for $n \geq 2$ and $s \in (0,1)$. This is contrast to \cite[estimate (7.15)]{MR4358465}, which holds for $n \in \{2,3\}$ and $s>n/4$. Second, is contrast to \cite[Theorem 7.5]{MR4358465}, where $\partial \Omega \in C^{\infty}$, in Theorem \ref{thm:error_estimate_control} we assume that $\Omega$ is merely Lipschitz.}
\end{remark}

\begin{remark}[error bound \eqref{eq:error_estimate_in L2}]
\rm
\EO{If $s \geq \frac{1}{2}$, the bound \eqref{eq:error_estimate_in L2} reads $\| \bar z - \bar z_{\T} \|_{L^2(\Omega)} \lesssim h_{\T}|\log h_{\T}|^{2\upsilon}$, which is \emph{nearly--optimal} in terms of approximation. If $s < \frac{1}{2}$, \eqref{eq:error_estimate_in L2} reads $\| \bar z - \bar z_{\T} \|_{L^2(\Omega)} \lesssim h^{2s}|\log h_{\T}|^{2\upsilon}$, which is \emph{suboptimal} in terms of approximation.}
\end{remark}

We \EO{now present error bounds for $\bar u - \bar u_{\T}$ and $\bar p - \bar p_{\T}$.}

\begin{theorem}[error estimates]
Let the assumptions of Theorem \ref{thm:error_estimate_control} hold. Then, there exist $h_{\star} >0$ \EO{such that
\begin{equation}
\begin{aligned}
\| \bar u - \bar u_{\T} \|_{s} \lesssim h_{\T}^{\vartheta} |\log h_{\T}|^{\upsilon},
\qquad
& \| \bar u - \bar u_{\T}\|_{L^2(\Omega)} \lesssim h_{\T}^{2\vartheta} |\log h_{\T}|^{2\upsilon},
\\
 \| \bar p - \bar p_{\T} \|_{s} \lesssim h_{\T}^{\vartheta} |\log h_{\T}|^{\upsilon},
\qquad
& \| \bar p - \bar p_{\T} \|_{L^2(\Omega)} \lesssim h_{\T}^{2\vartheta} |\log h_{\T}|^{2\upsilon},
\end{aligned}
\end{equation}
for every $h_{\T} \leq h_{\star}$.}
\label{thm:error_estimate_state_adjoint}
\end{theorem}
\begin{proof}
\EO{The bound for $\| \bar u - \bar u_{\T} \|_{s}$ is contained in the proof of Theorem \ref{thm:error_estimates_adjoint_equation_2}:  
\[
\| \bar u - \bar y  \|_{s} \lesssim \| \bar z - \bar z_{\T} \|_{L^2(\Omega)} \lesssim h^{2\vartheta}_{\T} |\log h_{\T}|^{2\upsilon},
\qquad
\| \bar y - \bar{u}_{\T} \|_{s} \lesssim h_{\T}^{\vartheta} |\log h_{\T}|^{\upsilon},
\]
upon utilizing \eqref{eq:error_estimate_in L2}. Thus, $\| \bar u - \bar u_{\T} \|_{s} \lesssim h_{\T}^{\vartheta} |\log h_{\T}|^{\upsilon}$. The bound for $\| \bar u - \bar u_{\T}\|_{L^2(\Omega)}$ follows similar arguments upon utilizing $\| \bar y - \bar{u}_{\T}\|_{L^2(\Omega)} \lesssim h^{2\vartheta}_{\T} |\log h_{\T}|^{2\upsilon}$. The bounds for the error committed in the approximation of $\bar p$ are the content of Theorem \ref{thm:error_estimates_adjoint_equation_2}.}
\end{proof}

\subsection{The variational discretization approach}
\label{sec:variational_discretization_approach}
In this section, we propose a semidiscrete scheme that is based on the so-called \emph{variational discretization approach} \cite{Hinze:05}. The scheme, which involves discretization \emph{only} on the state space ($\mathbb{Z}_{ad}$ is not discretized), reads as follows: Find
$
\min \{ J(u_{\T},\mathsf{z}): (u_{\T},\mathsf{z}) \in \mathbb{V}(\T) \times \mathbb{Z}_{ad} \}
$
subject to
\begin{equation}\label{eq:weak_st_eq_discrete_va}
\mathcal{A}( u_{\T}, v_{\T})+\int_{\Omega} a(x,u_{\T}(x)) v_{\T}(x) \mathrm{d}x = \int_{\Omega} \mathsf{z}(x) v_{\T}(x) \mathrm{d}x \quad \forall v_{\T} \in \mathbb{V}(\T).
\end{equation}

The existence of a solution and first order optimality conditions follow standard arguments. In particular, if $\bar{\mathsf{z}}$ denotes a local minimum, then
\begin{equation}
 j_{\T}'(\bar{\mathsf{z}})(\mathsf{z} - \bar{\mathsf{z}})
= 
 (\bar p_{\T} + \alpha \bar{\mathsf{z}}, \mathsf{z} - \bar{\mathsf{z}})_{L^2(\Omega)}  \geq 0 
 \quad
 \forall \mathsf{z} 
 \in \mathbb{Z}_{ad},
 \label{eq:variational_inequality_discrete_va}
\end{equation}
where $\bar{p}_{\T} \in \mathbb{V}(\T)$ solves \eqref{eq:optimal_adjoint_state_discrete} with $\bar u_{\T} = S_{\T} \bar{\mathsf{z}}$, i.e., $\bar u_{\T}$ solves \eqref{eq:weak_st_eq_discrete_va} with $\mathsf{z}$ replaced by $\bar{\mathsf{z}}$. We notice that, in view of \eqref{eq:variational_inequality_discrete_va}, the following projection formula holds \cite[section 4.6]{MR2583281}:
$
\bar{\mathsf{z}}(x) = \Pi_{[ \mathfrak{a}, \mathfrak{b} ]}(-\alpha^{-1} \bar{p}_{\T}(x)) 
$
for a.e.~$x \in \Omega$. The scheme thus induces a discretization of optimal controls by projecting $\bar{p}_{\T}$ into the admissible control set. Since $\bar{\mathsf{z}}$ implicitly depends on $\T$, in what follows we adopt the notation $\bar{\mathsf{z}}_{\T}$.

\subsubsection{Convergence of discretizations}
\label{sec:convergence_semidiscrete} 

\EO{Within the setting of Theorem \ref{thm:convergence}, we present the following convergence results:
\begin{itemize}
 \item Let $\T \in \mathbb{T}$ and let $\bar{\mathsf{z}}_{\T} \in \mathbb{Z}_{ad}$ be a global solution of the semidiscrete scheme. Then, there exist nonrelabeled subsequences of $\{\bar{\mathsf{z}}_{\T}\}$ such that $\bar{\mathsf{z}}_{\T} \mathrel{\ensurestackMath{\stackon[1pt]{\rightharpoonup}{\scriptstyle\ast}}} \bar{z}$ in $L^{\infty}(\Omega)$ as $h_{\T} \downarrow 0$ and \eqref{eq:convergence} holds; $\bar{z}$ is a global solution to \eqref{eq:min_integral}--\eqref{eq:weak_st_eq_integral}.
 \item Let $\bar{z} \in \mathbb{Z}_{ad}$ be a strict local minimum of \eqref{eq:min_integral}--\eqref{eq:weak_st_eq_integral}. Then, there exists a sequence of local minima $\{\bar{\mathsf{z}}_{\T}\}$ of the semidiscrete scheme satisfying \eqref{eq:convergence}.
\end{itemize}
The proof of these results follow very similar arguments to the ones in the proofs of Theorems \ref{thm:convergence} and \ref{thm:convergence_local_minima}. For brevity, we do not present further details.}

\subsubsection{Error estimates}

Let $\{ \bar{\mathsf{z}}_{\T} \} \subset \mathbb{Z}_{ad}$ be a sequence of local minima such that $\bar{\mathsf{z}}_{\T} \rightarrow \bar z$ in $L^2(\Omega)$ as $h_{\T} \downarrow 0$; $\bar z \in \mathbb{Z}_{ad}$ being a local solution of \eqref{eq:min_integral}--\eqref{eq:weak_st_eq_integral}. In what follows, we derive a bound for $ \bar z - \bar{\mathsf{z}}_{\T}$ in $L^2(\Omega)$. The following result is instrumental.
\begin{theorem}[instrumental error bound]
\EO{Let $n \geq 2$ and $s \in (0,1)$. Let $\Omega$ be a bounded Lipschitz domain. Assume that \ref{A1}--\ref{A3}, \ref{B1}--\ref{B2}, \eqref{eq:assumptions_on_a_and_L_in_L2}, and \eqref{eq:u_h_bounded} hold. Let $\bar z \in \mathbb{Z}_{ad}$ satisfy the conditions \eqref{eq:second_order_equivalent}. Then, there exists $h_{\star}>0$ such that}
\begin{equation}
\label{eq:basic_estimate_va}
\mathfrak{C}\| \bar z - \bar{\mathsf{z}}_{\T} \|^2_{L^2(\Omega)} \leq \left[ j'(\bar{\mathsf{z}}_{\T}) - j'(\bar z) \right](\bar{\mathsf{z}}_{\T} - \bar z)
\quad
\forall 
h_{\T} \leq h_{\star},
\quad
\mathfrak{C} = 2^{-1}\min \{ \EO{\nu}, \alpha \},
\end{equation}
where $\alpha$ is the regularization parameter and $\EO{\nu}$ is the constant appearing in \eqref{eq:second_order_equivalent}.
\label{thm:instrumental_error_estiamate_va}
\end{theorem}
\begin{proof}
Define $v_{\T} = (\bar{\mathsf{z}}_{\T} - \bar z)/\| \bar{\mathsf{z}}_{\T} - \bar z \|_{L^2(\Omega)}$. We assume that (up to a subsequence if necessary) $v_{\T} \rightharpoonup v$ in $L^2(\Omega)$ as $h_{\T} \downarrow 0$. We prove that $v \in C_{\bar z}$, where $C_{\bar z}$ is defined in \eqref{eq:Cz}. \EO{Observe that $v_{\T}$ satisfies \eqref{eq:sign_cond} because $\bar{\mathsf{z}}_{\T} \in \mathbb{Z}_{ad}$. Since the set of elements satisfying \eqref{eq:sign_cond} is  weakly closed in $L^2(\Omega)$, we conclude that $v$ satisfies \eqref{eq:sign_cond} as well. We now prove that $\bar{\mathfrak{p}}(x) \neq 0 \implies v(x) = 0$ for a.e.~$x \in \Omega$; recall that} $\bar{\mathfrak{p}}= \bar p + \alpha \bar{z}$. Define $\bar{\mathfrak{p}}_{\T}(x):= \bar p_{\T}(x) + \alpha \bar{\mathsf{z}}_{\T}(x)$. \EO{The convergence property \eqref{eq:convergence property} and Theorem \ref{thm:error_estimates_adjoint_equation_2} allow us to conclude that $\bar{\mathsf{z}}_{\T} \rightarrow \bar z$ in $L^2(\Omega)$ guarantee that $\bar{\mathfrak{p}}_{\T} \rightarrow \bar{\mathfrak{p}}$ in $L^2(\Omega)$ as $h_{\T} \downarrow 0$. Invoke \eqref{eq:variational_inequality_discrete_va} with $\mathsf{z} = \bar{z}$ to thus conclude that}
\[
\int_{\Omega} \bar{\mathfrak{p}}(x)v(x) \mathrm{d}x = \lim_{h_{\T} \downarrow 0} \frac{1}{\| \bar{\mathsf{z}}_{\T} - \bar z \|_{L^2(\Omega)}} \left[ \int_{\Omega} [\bar p_{\T}(x) + \alpha \bar{\mathsf{z}}_{\T}(x)] [\bar{\mathsf{z}}_{\T}(x) - \bar z(x)]\mathrm{d}x  \right] \leq 0.
\]
On the other hand, since $v$ satisfies \eqref{eq:sign_cond}, we have that $\bar{\mathfrak{p}}(x)v(x) \geq 0$ and thus that $\int_{\Omega} \bar{\mathfrak{p}}(x)v(x) \mathrm{d}x  = 0$. Consequently, $\bar{\mathfrak{p}}(x) \neq 0$ implies that $v(x) = 0$ for a.e.~$x \in \Omega$. We have thus proved that $v \in C_{\bar{z}}$.

We \EO{now proceed to derive \eqref{thm:instrumental_error_estiamate_va}. To accomplish this task, we write}
\begin{equation}
( j'(\bar{\mathsf{z}}_{\T}) - j'(\bar z) )(\bar{\mathsf{z}}_{\T} -\bar z )
=
j''(  \hat{\mathsf{z}}_{\T} ) (\bar{\mathsf{z}}_{\T} -\bar z )^2,
\qquad
\hat{\mathsf{z}}_{\T}:=
\bar z + \theta_{\T}(\bar{\mathsf{z}}_{\T} -\bar z ),
\label{eq:mean_value_theorem}
\end{equation}
where $\theta_{\T} \in (0,1)$. Let $u_{\hat{\mathsf{z}}_{\T}}$
solve \eqref{eq:weak_st_eq_integral} with $z$ replaced by $\hat{\mathsf{z}}_{\T}$, and let $p_{\hat{\mathsf{z}}_{\T}}$ be the solution to \eqref{eq:adj_eq_integral} with $u$ replaced by $u_{\hat{\mathsf{z}}_{\T}}$. Since $\bar{\mathsf{z}}_{\T} \rightarrow \bar z$ in $L^2(\Omega)$ as $h_{\T} \downarrow 0$ and \EO{$\{ \bar{\mathsf{z}}_{\T} \}$ is uniformly bounded in $L^{\infty}(\Omega)$}, we deduce that
$
u_{\hat{\mathsf{z}}_{\T}} \rightarrow \bar u
$
and
$
p_{\hat{\mathsf{z}}_{\T}} \rightarrow \bar p
$
in $\tilde H^s(\Omega) \cap L^{\infty}(\Omega)$ as $h_{\T} \downarrow 0$; \EO{see the proof of Theorem \ref{thm:suff_opt_cond}}. Similarly, we have that $\phi_{\T} := S'(\hat{\mathsf{z}}_{\T})v_{\T} \rightharpoonup S'(\bar z) v =: \phi$ in $\tilde H^s(\Omega)$ as $h_{\T} \downarrow 0$.
We thus obtain that
\begin{multline*}
\lim_{h_{\T} \rightarrow 0}
j''(  \hat{\mathsf{z}}_{\T} ) v_{\T}^2
 =
\lim_{h_{\T} \rightarrow 0} \int_{\Omega} \left[ \frac{\partial^2 L}{\partial u^2}(x,u_{\hat{\mathsf{z}}_{\T}})
\phi_{\T}^2
-
p_{\hat{\mathsf{z}}_{\T}} \frac{\partial^2 a}{\partial u^2}(x,u_{\hat{\mathsf{z}}_{\T}})
\phi_{\T}^2
+
\alpha v_{\T}^2
\right] 
\mathrm{d}x
\\
=
\alpha
+
\int_{\Omega} \left[ \frac{\partial^2 L}{\partial u^2}(x,\bar u)
\phi^2
-
\bar p \frac{\partial^2 a}{\partial u^2}(x,\bar u)
\phi^2
\right] 
\mathrm{d}x
=
\EO{\alpha + j''(\bar z)v^2 - \alpha \|v  \|_{L^2(\Omega)}^2}.
\end{multline*}
Since $\bar z$ satisfies \eqref{eq:second_order_equivalent}, we have
$\EO{\lim_{h_{\T} \rightarrow 0} j''(  \hat{\mathsf{z}}_{\T} ) v_{\T}^2 \geq \min \{ \EO{\nu}, \alpha \}}$ upon utilizing that $\| v \|_{L^2(\Omega)} \leq 1$. This yields the existence of $h_{\ast}>0$ such that
$
j''(  \hat{\mathsf{z}}_{\T} ) v_{\T}^2 \geq 2^{-1} \min \{ \EO{\nu}, \alpha \}
$
for $h_{\T} \leq h_{\ast}$. In view of the definition of $v_{\T}$ and \eqref{eq:mean_value_theorem}, we thus arrive at  \eqref{eq:basic_estimate_va}.
\end{proof}

We are now ready to derive a bound for the error $\bar z- \bar{\mathsf{z}}_{\T}$ in $L^2(\Omega)$.

\begin{theorem}[error estimate]
\EO{Let the assumptions of Theorem \ref{thm:instrumental_error_estiamate_va} hold. 
Assume, in addition, that $\partial L/\partial u = \partial L/\partial u(x,u)$ is locally Lipschitz in $u$. Let $\bar{z} \in \mathbb{Z}_{ad}$ satisfy the second order conditions \eqref{eq:second_order_equivalent}. Then, there exists $h_{\star} >0$ such that}
\begin{equation}
\label{eq:error_estimate_in L2_va}
\| \bar z - \bar{\mathsf{z}}_{\T} \|_{L^2(\Omega)}\lesssim \EO{h_{\T}^{2\vartheta} |\log h_{\T}|^{2\upsilon}}
\qquad
\forall h_{\T} \leq h_{\star}
\qquad
\EO{\vartheta = \min \{ s, \tfrac{1}{2} \}}.
\end{equation}
Here, \EO{$\upsilon = \tfrac{1}{2}$ if $s > \tfrac{1}{2}$, $\upsilon = \tfrac{3}{2} + \zeta$ if $s = \tfrac{1}{2}$, and $\upsilon = \tfrac{1}{2} + \zeta$ if $s < \tfrac{1}{2}$; $\zeta$ is as in the statement of Proposition \ref{pro:state_regularity_Lipschitz_new}. The hidden constant is independent of $h_{\T}$.}
\label{thm:error_estimate_control_va}
\end{theorem}
\begin{proof}
We invoke the instrumental estimate \eqref{eq:basic_estimate_va}, the continuous variational inequality \eqref{eq:var_ineq_integral}, and the semidiscrete one \eqref{eq:variational_inequality_discrete_va} to arrive at
\[
\mathfrak{C}\| \bar z - \bar{\mathsf{z}}_{\T} \|^2_{L^2(\Omega)} \leq \left[ j'(\bar{\mathsf{z}}_{\T}) - j_{\T}'(\bar{\mathsf{z}}_{\T}) \right](\bar{\mathsf{z}}_{\T} - \bar z).
\]
We now observe that $(j'(\bar{\mathsf{z}}_{\T}) - j_{\T}'(\bar{\mathsf{z}}_{\T}) )(\bar{\mathsf{z}}_{\T} - \bar z) = (p(\bar{\mathsf{z}}_{\T}) - \bar p_{\T}, \bar{\mathsf{z}}_{\T} - \bar z)_{L^2(\Omega)}$. Here, $\bar p_{\T}$ solves \eqref{eq:optimal_adjoint_state_discrete} and $p(\bar{\mathsf{z}}_{\T})$ denotes the solution to \eqref{eq:adj_eq_integral} with $u$ being the solution to \eqref{eq:weak_st_eq_integral} with $z$ replaced by $\bar{\mathsf{z}}_{\T}$. Similar arguments to the ones elaborated within the proofs of Theorems \ref{thm:error_estimates_adjoint_equation} and \ref{thm:error_estimates_adjoint_equation_2} can be utilized to obtain $\| p(\bar{\mathsf{z}}_{\T}) - \bar p_{\T}\|_{L^2(\Omega)} \lesssim \EO{h^{2\vartheta}_{\T} |\log h_{\T}|^{2\upsilon}}$. This bound implies the desired error estimate \eqref{eq:error_estimate_in L2_va} and concludes the proof.
\end{proof}

\subsubsection{Error estimates on graded meshes}

In this section, we operate under the family of graded meshes $\{ \T \}$ of $\bar \Omega$ dictated by \eqref{eq:graded_meshes_state_equation} and obtain a bound for $\bar z- \bar{\mathsf{z}}_{\T}$ in $L^2(\Omega)$, \EO{which improves the one derived for the fully discrete scheme in Theorem \ref{thm:error_estimate_control} and the one obtained for the semidiscrete scheme in Theorem \ref{thm:error_estimate_control_va}.}

\begin{theorem}[\EO{error bound on graded meshes}]
Let the assumptions of Theorem \ref{thm:error_estimate_control_va} hold. Assume that $\Omega$ satisfies, in addition, an exterior ball condition. Let $\mu = n/(n-1)$ be the parameter that dictates \eqref{eq:graded_meshes_state_equation}, and let $\beta_{\star} = n/(2(n-1)) - s$. Assume that $a$, $\partial a/\partial u$, and $\partial L/\partial u$ satisfy \eqref{eq:Nemitskii} with $\varrho = \beta_{\star}$. If $a(\cdot,0) \in L^{\infty}(\Omega)$ and $\bar z_{\T}, \bar u_{\T} \in C^{\beta_{\star}}(\bar \Omega)$, uniformly with respect to discretization, then there exists $h_{\nabla} >0$ such that
\begin{equation}
\label{eq:error_estimate_in L2_va_graded}
\| \bar z - \bar{\mathsf{z}}_{\T} \|_{L^2(\Omega)}\lesssim h^{\tfrac{n}{2(n-1)} + \vartheta} |\log h|^{\upsilon},
\qquad
\vartheta = \min \{s,\tfrac{1}{2} \},
\qquad
\EO{\left[\frac{n}{4(n-1)},1\right)}.
\end{equation}
for all $h \leq h_{\nabla}$. Here, \EO{$\upsilon = \tfrac{3}{2}$ if $s > \tfrac{1}{2}$, $\upsilon = \tfrac{5}{2} + \zeta$ if $s = \tfrac{1}{2}$, and $\upsilon = \tfrac{3}{2} + \zeta$ if $s < \tfrac{1}{2}$; $\zeta$ is as in the statement of Proposition \ref{pro:state_regularity_Lipschitz_new}. The hidden constant is independent of $h$.}
\label{thm:error_estimate_control_va_graded}
\end{theorem}
\begin{proof}
\EO{We begin the proof by by considering $h$ sufficiently small such that $h_{\T} \leq h_{\star}$, where $h_{\star}$ is as in \eqref{eq:basic_estimate_va}. Consequently, we have at hand the estimate in \eqref{eq:basic_estimate_va}. We now follow the proof of Theorem \ref{thm:error_estimate_control_va} and observe that it suffices to bound $ p(\bar{\mathsf{z}}_{\T}) - \bar p_{\T}$ in $L^2(\Omega)$. To accomplish this task, we define $r \in \tilde H^s(\Omega)$ as the solution to}
\begin{equation}\label{eq:r}
\mathcal{A}(v,r) + \left(\frac{\partial a}{\partial u}(\cdot,\bar u_{\T})r,v\right)_{L^2(\Omega)} = \left(\frac{\partial L}{\partial u}(\cdot,\bar u_{\T}),v\right)_{L^2(\Omega)}
\quad \forall v \in \tilde H^s(\Omega).
\end{equation}
\EO{We recall that $p(\bar z_{\T})$ denotes the solution to \eqref{eq:adj_eq_integral} with $u$ replaced by $\mathcal{S} \bar z_{\T}$ and write the problem that $p(\bar{\mathsf{z}}_{\T}) - r$ solves. Observe that \ref{A3}, \eqref{eq:u_h_bounded}, and \eqref{eq:assumptions_on_a_and_L_in_L2} yield
\[
\mathfrak{L}:=
\frac{\partial a}{\partial u} (\cdot, \bar u_{\T} ) -\frac{\partial a}{\partial u} (\cdot,  S \bar z_{\T})\in L^2(\Omega),
\quad
\mathfrak{N}:=
 \frac{\partial L}{\partial u} (\cdot, S \bar z_{\T} ) -\frac{\partial L}{\partial u} (\cdot, \bar u_{\T}) \in L^2(\Omega),
\]
We thus invoke a basic stability bound for the aforementioned problem on the basis of assumption \ref{A2} and the fact that $\mathfrak{L}$ and $\mathfrak{N}$ belong to $L^2(\Omega)$, uniformly with respect to discretization, and utilize the fact that $\partial a/\partial u = \partial a/\partial u(x,u)$ and $\partial L/\partial u = \partial L/\partial u(x,u)$ are locally Lipschitz in $u$ to obtain
 \begin{equation*}
\| p(\bar{\mathsf{z}}_{\T}) - r\|_{s} 
\lesssim 
\| \mathfrak{N} \|_{L^2(\Omega)}
+
\| \mathfrak{L} \|_{L^2(\Omega)}
\lesssim 
\| S \bar z_{\T}  - \bar u_{\T} \|_{L^2(\Omega)}.
\end{equation*}
We now derive a bound for $\| S \bar z_{\T}  - \bar u_{\T} \|_{L^2(\Omega)}$.}
Since $\{ \bar z_{\T} \} \subset \mathbb{Z}_{ad}$, $a(\cdot,0) \in L^{\infty}(\Omega)$, and $\Omega$ satisfies an exterior ball condition, Proposition \ref{pro:state_regularity_Holder} reveals that $S \bar z_{\T} \in C^s(\bar \Omega)$. Utilize now that $a$ satisfies \eqref{eq:Nemitskii} with $\varrho = \beta_{\star}$ and that $\bar z_{\T} $ belongs to $C^{\beta_{\star}}(\bar \Omega)$, uniformly with respect to discretization, to obtain $\bar z_{\T} - a(\cdot,S \bar z_{\T}) \in C^{\beta_{\star}}(\bar \Omega)$, upon using the results of Remark \ref{rk:Nemitskii} and the inequality $\beta_{\star} = n/(2(n-1)) -s \leq s$. Apply \eqref{eq:graded_estimate_2} to arrive at
$
\| S \bar z_{\T}  - \bar u_{\T} \|_{L^2(\Omega)} \lesssim h^{\EO{\mathfrak{0}}} | \log h|^{\upsilon},
$
where \EO{$\mathfrak{o} = n/2(n-1) + \vartheta$}. To bound $r - \bar p_{\T}$, we observe that $\bar p_{\T}$ is the finite element approximation of $r$ within $\mathbb{V}(\T)$ and use \cite[Proposition 3.10]{MR4283703}:
$
\| r - \bar p_{\T} \|_{L^2(\Omega)} \lesssim h^{\EO{\mathfrak{o}}} | \log h|^{\upsilon}.
$
Notice that, $\partial a/\partial u$ and $\partial L/\partial u$ satisfies \eqref{eq:Nemitskii} with $\varrho = \beta_{\star}$ and $\bar u_{\T} \in C^{\beta_{\star}}(\bar \Omega)$. Thus, $\partial L/\partial u (\cdot,\bar u_{\T}) - \partial a/\partial u(\cdot,\bar u_{\T})r \in C^{\beta_{\star}}(\bar \Omega)$ uniformly with respect to discretization. Observe that $r \in C^{s}(\bar \Omega)$ and $s \geq \beta_{\star}$.
\end{proof}
\\
\\
\textbf{Declarations}
\\
\textbf{Competing Interests:} The author has not disclosed any competing interests.

\bibliographystyle{plain}
\bibliography{bibliography}

\end{document}